\documentclass[12pt]{report}

\usepackage[margin=1in, headheight=14pt, headsep=0.5in, footskip=0.5in]{geometry}
\usepackage{setspace}
\usepackage{fancyhdr}
\raggedright
\usepackage{graphicx}       
\usepackage{tocloft}        
\usepackage{setspace}       
\usepackage{titlesec}       
\usepackage{indentfirst}

\titleformat{\chapter}
  {\normalfont\bfseries\centering\Huge}
  {Chapter \thechapter : }
  {0em}
  {}                        

\titlespacing*{\chapter}{0pt}{0pt}{1\baselineskip}

\titleformat{\section}
  {\normalfont\bfseries\centering\Large}
  {\Large \thesection}  
  {1em}          
  {}

\titlespacing*{\section}{0pt}{0pt}{1\baselineskip}

\usepackage{amsmath, amssymb, amsthm, changepage, faktor, float, graphicx, caption, subcaption, thmtools, thm-restate, hyperref, cleveref, array, multirow, rotating, hhline}

\usepackage[ruled,vlined]{algorithm2e}

\newtheorem{thm}{Theorem}[chapter]
\newtheorem{lem}[thm]{Lemma}
\newtheorem{cor}[thm]{Corollary}
\newtheorem{prop}[thm]{Proposition}
\theoremstyle{definition}
\newtheorem{remark}[thm]{Remark}

\def\ZZ{{\mathbb Z}}

\def\im{{\textnormal{im}}}

\def\rca{{\textnormal{rca}}}

\def\r{{\tilde{r}}}
\def\s{{\tilde{s}}}

\title{\Huge On Clique Graphs and Clique Regular Graphs \\[1em]
\Large  Honors Capstone Project}
\author{Connor Phillips \\[0.5em]
James Madison University\\
Department of Mathematics and Statistics}
\date{\today}

\begin{document}

\begin{titlepage}
\centering

{\large\bfseries
A Comprehensive Study of Clique Graphs and Clique Regular Graphs
\par}

\vspace{0.1in}
\noindent\rule{2.5in}{0.4pt}

\vspace{0.1in}

{\doublespacing An Honors College Project Presented to\\
the Faculty of the Undergraduate\\
College of Science and Mathematics\\
James Madison University\\}

\vspace{0.1in}
\noindent\rule{2.5in}{0.4pt}

\vspace{0.2in}

by Connor M. Phillips

\vspace{0.1in}

December 2025\\[0.2in]

\noindent\rule{\textwidth}{1.5pt}

{\raggedright\footnotesize\doublespacing Accepted by the Faculty of the Department of Mathematics and Statistics,
James Madison University, in partial fulfillment of the requirements
for the Honors College.\par}

\vspace{0.15in}

\begin{center}
\setlength{\tabcolsep}{15pt}
\renewcommand{\arraystretch}{0.75}
\begin{tabular}{p{2.5in} p{2.5in}}
{\scriptsize\raggedright FACULTY COMMITTEE\par} &  
{\scriptsize\raggedright HONORS COLLEGE APPROVAL\par} \\[0.2cm]

{\centering\rule{2in}{0.4pt}} & 
{\centering\rule{2in}{0.4pt}} \\
{\scriptsize\raggedright Project Advisor: Joshua Ducey, Ph.D.} & 
{\scriptsize\raggedright Bethany Blackstone, Ph.D.} \\
{\scriptsize\raggedright Professor, Pure Math} & 
{\scriptsize\raggedright Dean, Honors College} \\[0.2cm]

{\centering\rule{2in}{0.4pt}} & \\
{\scriptsize\raggedright Reader: Rebecca Field, Ph.D.} & \\
{\scriptsize\raggedright Associate Professor, Pure Math} & \\[0.2cm]

{\centering\rule{2in}{0.4pt}} & \\
{\scriptsize\raggedright Reader: Brant Jones, Ph.D.} & \\
{\scriptsize\raggedright Professor, Pure Math}  &
\end{tabular}
\end{center}

\vspace{0.15in}

\noindent\rule{\textwidth}{1.5pt}

{\raggedright\footnotesize\doublespacing PUBLIC PRESENTATION\par}
{\raggedright\footnotesize\doublespacing Selections of this work were presented at the Shenandoah Undergraduate Mathematics and Statistics Conference on November 2\textsuperscript{nd}, 2024 and November 1\textsuperscript{st}, 2025.\par}

\end{titlepage}
\clearpage

\pagenumbering{arabic}
\setstretch{1.56}
\centering\tableofcontents
\thispagestyle{plain}
\clearpage

\doublespacing
\listoffigures
\thispagestyle{plain}
\clearpage

\raggedright
\setlength{\parindent}{2em}


\chapter*{Acknowledgments}
\addcontentsline{toc}{chapter}{Acknowledgments}
This research was conducted under the supervision of Dr.\ Ducey, whose guidance and mentorship over the past two years have been invaluable. Appreciation is also extended to Dr.\ Field and Dr.\ Jones for their careful review and constructive feedback on this project. Collaboration with Robert Petro during the summer of 2024 contributed to several aspects of the research, and the results in Chapters 1,2,3,6 and section 4.1 have been published with Robert Petro in Discrete Mathematics.
\clearpage

\chapter*{Abstract}
\addcontentsline{toc}{chapter}{Abstract}
If $\Gamma$ is a graph for which every edge is in exactly one clique of order $\omega$, then one can form a new graph with vertex set equal to these cliques.  This is a generalization of the line graph of $\Gamma$.  We discover many general results and classifications related to these clique graphs that will be useful to researchers studying graphs with this property. In particular, we find bounds on the spectrum of $\Gamma$ (with exact results when $\Gamma$ is $k$-regular) and some complete classifications when $\Gamma$ is strongly regular. We apply our results to derive novel information about the existence questions of certain strongly regular graphs. We also examine the critical group of graphs with this property and their associated transformations. Finally, we discuss examples of widely studied families of graphs that have this property, and provide some examples to make the results more concrete. 
\clearpage

\chapter{Introduction}
A famous question in the field of graph theory is: Does there exist a graph on 99 vertices that is 14-regular such that every pair of adjacent vertices is contained in a unique triangle and every pair of non-adjacent vertices is contained in a unique quadrangle? This problem has had a \$1,000 bounty put on a solution by famous mathematician John Conway since 2017, which has yet to be claimed. This question is known as the Conway-99 graph problem and while it is widely studied, there is one approach to this problem that has yet to be explored with much intensity.\par
If such a graph existed, each edge would be contained in exactly one of its $\frac{99\cdot14}{6}=231$ triangles, and so one could construct a new graph by taking those triangles as vertices, and setting two triangles adjacent if and only if they share a vertex. Such a graph would have properties very closely related to those of the original graph, and so could be used in either a construction of the graph or a proof of its non-existence.\par
In general, graphs with the property that each edge is contained in a unique triangle are called \textbf{locally linear}. In this report, we will introduce a generalization of locally linear graphs and a graph transformation that seems very natural to do with graphs that have this property. We will provide many results that relate to these types of graphs and their associated transformation, including results pertaining to the graphs' spectrums and critical groups, and classifications relating to the line graph and strongly regular graphs. Finally, we will discuss concrete examples of families of graphs that have this property which are widely studied in the field of graph theory.\par
But first, we must introduce graph theory and some notation. A \textbf{graph} $\Gamma=(V,E)$ is a finite vertex set $V$, and a finite edge set $E$ containing unordered pairs of vertices. When we write $v\in \Gamma$, it is understood that this means $v$ is a vertex of $\Gamma$. The graphs we discuss here will have neither loops nor parallel edges. We write an edge between two vertices $x$ and $y$ as either $\{x,y\}$ or $xy$. Two vertices are adjacent if there is an edge between them and we denote the neighborhood of vertex $v$ as $N(v)$, the set of vertices adjacent to $v$.\par
The number of vertices that $v$ is adjacent to is called the degree of the vertex and is denoted $d(v)=|N(v)|$. The largest degree of any vertex in $\Gamma$ is denoted $\Delta(\Gamma)$. If every vertex in a graph has the same degree $k$, we say the graph is \textbf{regular}, or more specifically \textbf{$k$-regular}. This is the first of many definitions that we will introduce in this report which include the word regular. We apologize for any confusion this may cause the reader, but all of these definitions are standard in the literature.\par
A \textbf{clique} of order $\omega$ is a subset of $\omega$ vertices such that each vertex is adjacent to every other vertex. The \textbf{clique number} of a graph, denoted $\omega(\Gamma)$, is the order of the largest clique in $\Gamma$. A clique is \textbf{maximal} if it is not contained in a larger clique, and a clique is \textbf{maximum} if it is of order $\omega(\Gamma)$.\par
The complete graph on $n$ vertices, $K_n$, is the graph with $n$ vertices such that each vertex is adjacent to every other vertex. A graph is called bipartite if the vertices can be partitioned into two sets, such that no two vertices in the same set are adjacent. The complete bipartite graph, $K_{n,m}$, is the graph on $n+m$ vertices which are partitioned into two sets of order $n$ and $m$ such that the graph is bipartite on that partition and any two vertices in different sets are adjacent.\par
The \textbf{line graph} of a graph $\Gamma$, $L(\Gamma)$, is the graph with vertices as the edges of $\Gamma$, where two edges are adjacent if and only if they share a vertex in $\Gamma$. We will define more terms and notation as we use them throughout this report. For any other definitions not specified, please see a standard graph theory reference \cite{harary}.\par
\clearpage
\section{Clique Graphs and Clique Regular Graphs}
We introduce an \textbf{$\omega$-clique regular} graph as a graph with a nonempty edge set such that every edge is in a unique clique of order $\omega$. Since a 2-clique is just a single edge every graph with a nonempty edge set is 2-clique regular. Notice also that since a $3$-clique is a triangle, the locally linear graphs are exactly the 3-clique regular graphs. For this reason we consider the clique regular property to be a generalization of the locally linear property, and taking $\omega=3$ for any of our results reveals a fact about locally linear graphs.\par
For graphs with the $\omega$-clique regular property, much of the structure of the graph can be understood by examining its $\omega$-cliques. We develop a very natural graph transformation that can be be applied to $\omega$-clique regular graphs, which reveals the structure of the connections between their $\omega$-cliques. This transformation generalizes the way in which the edges of a graph can be studied using the line graph transformation. \par
We introduce the \textbf{$\omega$-clique graph} of a graph $\Gamma$, which shows the adjacency between the cliques of order $\omega$ in $\Gamma$. This graph is denoted $C_\omega(\Gamma)$ and is defined in the following manner:
\begin{enumerate}
    \item The vertices of the clique graph $C_\omega(\Gamma)$ are the cliques of order $\omega$ in $\Gamma$.
    \item Two distinct cliques are adjacent in $C_\omega(\Gamma)$ if and only if they have a nonempty intersection.
\end{enumerate}
While the $\omega$-clique graph construction can be applied to any graph, it is most useful when the original graph is $\omega$-clique regular. For the most part in this report, we will restrict our focus only to considering the $\omega$-clique graph of $\omega$-clique regular graphs.\par
Since a $2$-clique is just a single edge, it is clear that for any graph $\Gamma$, its line graph is isomorphic to its $2$-clique graph, $C_2(\Gamma) \cong L(\Gamma)$. For this reason, we consider the clique graph construction to be a generalization of the line graph and taking $\omega =2$ for any of our theorems reveals a fact about the line graph of any nonempty graph. Figure \ref{fig:clique} shows an example of a graph that is 3-clique regular and its 3-clique graph.\par
\begin{figure}[htbp]
    \centering
    \begin{subfigure}[c]{0.45\textwidth}
        \centering
        \includegraphics[width=0.8\textwidth]{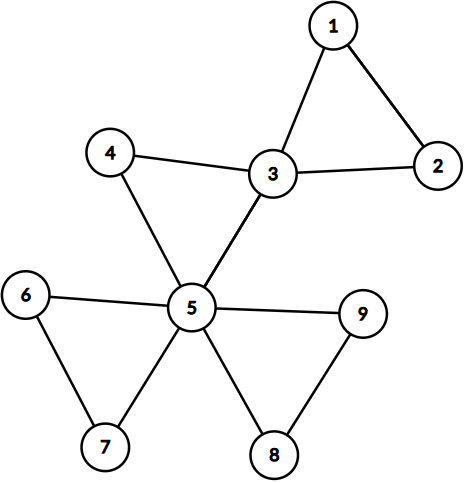}
        \caption{A 3-clique regular graph $\Gamma$}
    \end{subfigure}
    \begin{subfigure}[c]{0.45\textwidth}
        \centering
        \includegraphics[width=0.5\textwidth]{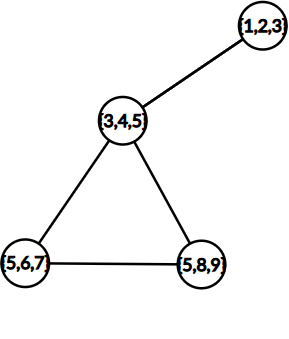}
        \caption{The graph $C_3(\Gamma)$}
    \end{subfigure}
    \caption{An example of a clique regular graph and its clique graph}
    \label{fig:clique}
\end{figure}
Our first proposition will enumerate for some graph $\Gamma$, what values of $\omega$ it could possibly be $\omega$-clique regular for. This will also justify our choice of the notation $\omega$.
\begin{prop}
    If $\Gamma$ is an $\omega$-clique regular graph, then $\omega=2$ or $\omega=\omega(\Gamma)$.
\end{prop}
\begin{proof}
    Assume $\Gamma$ is $\omega$-clique regular and $\omega \neq 2$. If $2 < \omega < \omega(\Gamma)$, then there exists a clique in $\Gamma$ of order $\omega + 1$ contained within a clique of order $\omega(\Gamma)$. Within this $(\omega +1)$-clique, removing any vertex generates an $\omega$-clique. So every edge is in at least the $\omega -1$ cliques of order $\omega$, that are generated by removing each vertex but the edge's endpoints. Since $\omega-1>1$, these edges are not in unique $\omega$-cliques. If $\omega(\Gamma) < \omega$, then there are no $\omega$-cliques in $\Gamma$, and so no edge is contained in an $\omega$-clique. So $\omega$ must equal $\omega(\Gamma)$.
\end{proof}
To avoid distracting notation throughout this report, we will refer to $\omega$-clique regularity and $\omega$-clique graphs generally without the quantifier $\omega$ when it is clear through context the value in question, or when the specific value is not relevant to topic of discussion.

Notice that the definition of clique regularity is similar to that of a \textbf{regular clique assembly}, as introduced in \cite{guest}. They define a regular clique assembly as a regular graph with clique number at least 2, where every maximal clique is maximum and each edge belongs to exactly one maximum clique. These constructions form a subset of clique-regular graphs, and most of the results in this report only require a graph to be clique-regular rather than a full regular clique assembly. However, in certain specific cases, we will refer to results from \cite{guest} to build upon their findings and establish stronger conclusions about regular clique assemblies. \par
Regular clique assemblies are related to \textbf{edge regular graphs}, denoted erg$(n,k,\lambda)$, which are defined as $k$-regular graphs on $n$ vertices with the property that every pair of adjacent vertices has $\lambda$ common neighbors. Clearly every regular clique assembly $\Gamma$ is $\omega(\Gamma)$-clique regular and the following theorem will classify when clique regular graphs are regular clique assemblies. 

\begin{thm} \label{thm:rca1}
    Suppose $\Gamma$ is an $\omega$-clique regular graph on $n$ vertices. Then $\Gamma$ is a regular clique assembly if and only if $\Gamma$ is an \textnormal{erg}$(n,\Delta(\Gamma),\omega -2)$.
\end{thm}
\begin{proof}
    The proof that any regular clique assembly is an erg$(n,\Delta(\Gamma),\omega -2)$ comes from \cite[p. 304]{guest}. So now assume that $\Gamma$ is $\omega$-clique regular and an erg$(n,k,\omega-2)$ and we will show that $\Gamma$ is a regular clique assembly. If $\omega =2$ or $\omega=3$ this proof also comes from \cite[p. 304]{guest} so assume that $\omega$ is the clique number of $\Gamma$ and is greater than 3. Since $\Gamma$ is regular from the hypothesis, it is sufficient to show that every maximal clique in $\Gamma$ has order $\omega$. Suppose for contradiction there exists a maximal clique of order less than $\omega$, call it $c_<$, and let $x$ and  $y$ be vertices in that clique. Then the edge $xy$ is in a unique clique of order $\omega$, call it $c_\omega$, so $x$ and $y$ have $\omega -2$ neighbors in that clique. Because $c_<$ is maximal, it is not contained in $c_\omega$ and there exists a vertex $z$ in $c_<$ and not in $c_\omega$. So $z$ is also a common neighbor of $x$ and $y$ which implies the order of $N(x)\cap N(y)$ is greater than $\omega -2$, a contradiction.
\end{proof}
We denote by rca$(n,k,\omega)$, a regular clique assembly on $n$ vertices that is $k$-regular with clique number $\omega$. An important result that we will use in section \ref{sec:localsrgs}, comes from the original paper on regular clique assemblies \cite{guest}, so we will record it here at the end of the introduction.
\begin{thm}\label{thm:rca2}
    If $\Gamma$ is an \textnormal{rca}$(n,k,\omega)$, then $C_\omega(\Gamma)$ is an \textnormal{rca}$\left(\frac{nk}{\omega(\omega -1)}, \omega \left(\frac{k}{\omega -1}-1 \right), \frac{k}{\omega -1}\right)$ and $C_{\frac{k}{\omega -1}}(C_\omega(\Gamma))\cong\Gamma$.
\end{thm}
The proof of this theorem comes from \cite[Theorem 1]{guest}.
Lastly, we introduce another lemma on regular clique assemblies that will be useful in section \ref{sec:localsrgs}.
\begin{lem}\label{lem:nonadj}
    Suppose $\Gamma$ is an \textnormal{rca}$(n,k,\omega)$ and let $u$ and $v$ be distinct and non-adjacent vertices in $\Gamma$. Then $|N(u)\cap N(v)|\leq \frac{k}{\omega -1}$.
\end{lem}
\begin{proof}
    If $\omega=2$, then the lemma is clear so assume $\omega \geq 3$. From \cite[Proposition 1]{guest}, the neighborhood of each vertex in $\Gamma$ is a disjoint union of $K_{\omega -1}$'s with $\frac{k}{\omega -1}$ components. Suppose $x$ and $y$ are distinct vertices in $N(v)\cap N(u)$. If $x$ and $y$ are in the same $K_{\omega -1}$ component of $N(v)$, then since $\Gamma$ is a regular clique assembly, the clique $\{x,y,u\}$ is contained in an $\omega$-clique containing $u$ and the clique $\{x,y,v\}$ is contained in a different $\omega$-clique containing $v$. But this is a contradiction since the edge $\{x,y\}$ is in two different $\omega$-cliques. So $x$ and $y$ must be in distinct $K_{\omega-1}$ components in $N(v)$ showing there exists an injective mapping from $N(v)\cap N(u)$ to the disjoint $K_{\omega-1}$ components in $N(v)$. This implies the result.
\end{proof}
\chapter{Basic Graph Theory}

In this section, we investigate which graphs have a clique regular line graph, and for which the clique graph acts as the ``inverse" of the line graph construction; which graphs $\Gamma$ satisfy $C_\omega(L(\Gamma))\cong \Gamma$. The following lemmas and theorems will classify for $\omega \geq 3$, all graphs which have the first property and all connected graphs that have the second property.
\begin{lem} \label{lem:d(v)}
    For each vertex $v$ in $\Gamma$, the set of edges incident to $v$ form a clique in $L(\Gamma)$ with order $d(v)$. If $d(v) > 2$, then this clique is maximal.
\end{lem}
\begin{proof}
    Vertex $v$ is incident to $d(v)$ edges in $\Gamma$ all of which share one endpoint $v$. So in $L(\Gamma)$ those edges incident to $v$ form a clique of order $d(v)$. Now assume this clique is not maximal and we will show $d(v) \leq 2$. The clique of the edges of $v$ must then be contained in a larger clique in $L(\Gamma)$, and so there must exist an edge $e$ in $\Gamma$ that is not incident to $v$ but is adjacent to all other edges incident to $v$. If $v$ had degree 3 or larger, by the pigeon hole principle one of the 2 endpoints of $e$ would have to be the endpoint for at least 2 edges incident to $v$. But we do not allow parallel edges, so $v$ can have degree at most 2.
\end{proof}
Following from Lemma \ref{lem:d(v)}, for any vertex $v \in \Gamma$ when we refer to the ``clique created by $v$" we mean the clique in $L(\Gamma)$ of order $d(v)$ induced by the edges in $\Gamma$ incident to $v$.
\begin{lem} \label{lem:whit}
    For $\omega>3$, $L(\Gamma) \cong K_\omega$ if and only if $\Gamma \cong K_{1,\omega}$. Also, $L(\Gamma) \cong K_3$ if and only if $\Gamma \cong K_3$ or $\Gamma \cong K_{1,3}$.
\end{lem}
\begin{proof}
    The graph $K_{1,\omega}$ has $\omega$ edges that all share a common end point, so $L(K_{1,\omega}) \cong K_\omega$ for $\omega \geq 3$. Also, it is clear that $L(K_3) \cong K_3$. The converse for both cases follows from the Whitney Isomorphism Theorem \cite{whitney}, which states that two graphs are isomorphic if and only if their line graphs are isomorphic, with the single exception of graphs $K_3$ and $K_{1,3}$.
\end{proof}
\begin{remark} \label{remark}
    Lemma \ref{lem:whit} implies that the only structures that can create a maximal clique in $L(\Gamma)$ are the set of edges incident to a vertex $v$, i.e. cliques created by a vertex $v$ following from Lemma \ref{lem:d(v)}. There is one exception in the case of a 3-clique in $L(\Gamma)$, which can also be created by a triangle, $K_3$.
\end{remark}
We will begin by classifying graphs that have $\omega$-clique regular line graphs and start with the case $\omega \geq 4$.
\begin{thm} \label{thm:4cr}
    Suppose $\omega \geq 4$ and $\Gamma$ is a graph. Then $L(\Gamma)$ is $\omega$-clique regular if and only if the degree of every vertex in $\Gamma$ is 1 or $\omega$.
\end{thm}
\begin{proof}
    First assume that the degree of every vertex in $\Gamma$ is 1 or $\omega$. Let $e=\{xv, yv\}$ be an edge in $L(\Gamma)$. Then the vertex $v$ in $\Gamma$ has degree $\omega$ and so by Lemma \ref{lem:d(v)} $e$ is in the maximal $\omega$-clique created by $v$, and by Remark \ref{remark}, this is the unique $\omega$-clique $e$ is in.\\
    For the converse, assume there exists a vertex $v$ in $\Gamma$ with degree not equal to either 1 or $\omega$. If $1 < d(v) < \omega$, then there exists edges $e_1$ and $e_2$ incident to $v$ in $\Gamma$. So the edge $e_1e_2$ in $L(\Gamma)$ is in the maximal clique created by $v$ of order less than $\omega$ and $e_1$ and $e_2$ share no other endpoints in $\Gamma$ so by Remark \ref{remark}, the edge $e_1e_2$ is in no clique of order $\omega$. If $d(v) > \omega $, then there exists a clique of order greater than $\omega$ in $L(\Gamma)$. So in either case $L(\Gamma)$ is not $\omega$-clique regular.
\end{proof}
The case of the $3$-clique regular line graph is more complicated because\linebreak $L(K_3) \cong K_3 \cong L(K_{1,3})$. For this reason, the concept of a triangle-free graph will be useful in our classification. This is defined as a graph that has no triangles as a subgraph, or equivalently as a graph $\Gamma$ with $\omega(\Gamma)=2$.
\begin{thm} \label{thm:3cr}
    Suppose $\Gamma$ is a graph with connected components $C_1, C_2, \ldots , C_n$. Then $L(\Gamma)$ is $3$-clique regular if and only if whenever $C_i \not \cong K_3$ , $C_i$ is triangle-free and the degree of every vertex in $C_i$ is 1 or $3$. 
\end{thm}
\begin{proof}
    First, assume that for each $C_i \not \cong K_3$, $C_i$ is triangle-free and the degree of every vertex in $C_i$ is 1 or $3$. Let $e=\{xv,yv\}$ be an edge in $L(\Gamma)$ with vertex $v$ in connected component $C_i$. If $C_i \cong K_3$, then edge $e$ is in $L(C_i) \cong K_3$ a connected component of $L(\Gamma)$ so we are done. So assume $C_i \not \cong K_3$ implying by our assumption that the degree of $v$ is $3$. So the $3$-clique created by $v$ is maximal by Lemma \ref{lem:d(v)} and since $C_i$ is triangle-free by our assumption, Remark \ref{remark} implies that this is the only $3$-clique containing edge $e$.\\
    Now assume the converse, that for some $C_i \not \cong K_3$, $C_i$ contains a triangle or $C_i$ contains a vertex with degree not equal to either $1$ or $3$. For the first case, since $C_i$ is connected, contains a triangle and is not $K_3$, Figure \ref{fig:sub1} is a subgraph of $\Gamma$ and therefore Figure \ref{fig:sub2} is a subgraph of $L(\Gamma)$. So $L(\Gamma)$ is not $3$-clique regular.\par
    \begin{figure}[htbp]
    \begin{subfigure}{.5\textwidth}
      \centering
      \includegraphics[width=.8\linewidth]{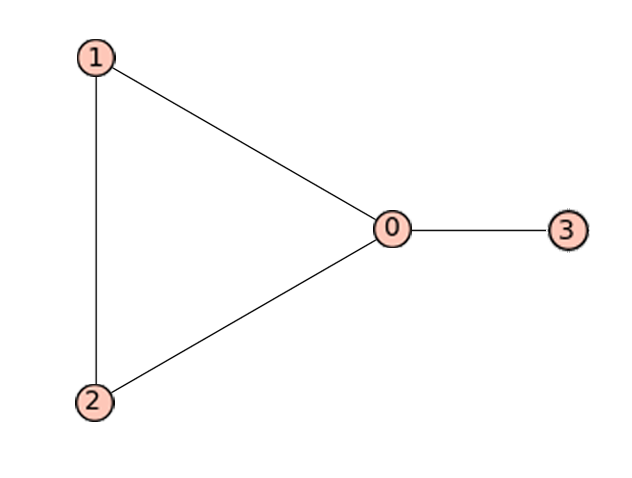}
      \caption{Subgraph of $\Gamma$}
      \label{fig:sub1}
    \end{subfigure}
    \begin{subfigure}{.5\textwidth}
      \centering
      \includegraphics[width=.8\linewidth]{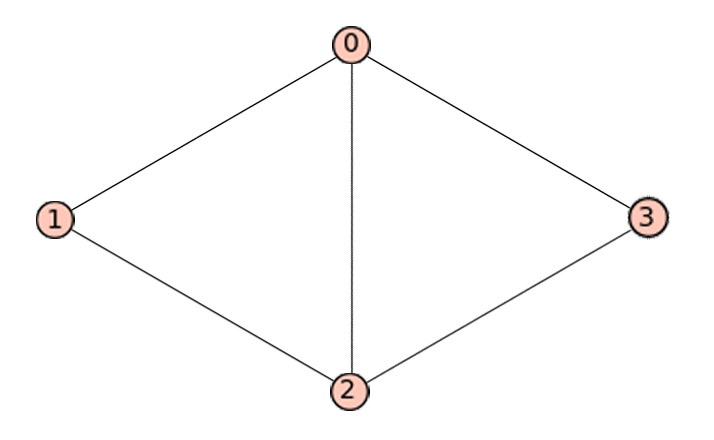}
      \caption{Subgraph of $L(\Gamma)$}
      \label{fig:sub2}
    \end{subfigure}
    \caption[Graphs used in the proof of Theorem \ref{thm:4cr}]{}
    \end{figure}
Now for the second case, let $v$ be a vertex in $C_i$ with degree not equal to either $1$ or $3$. If $d(v) > 3$, then there exists a clique of order greater than $3$ in $L(\Gamma)$ so $L(\Gamma)$ is not $3$-clique regular. If $d(v)=2$, then let vertices $x$ and $y$ be adjacent to $v$. If $x$ and $y$ are adjacent, then $\{v,x,y\}$ is a triangle in $\Gamma$, a contradiction. The edge $e = \{xv,yv\}$ in $L(\Gamma)$ is then not in a $3$-clique induced by the edges incident to $v$ and not in a $3$-clique induced by a triangle in $\Gamma$. So by Remark \ref{remark}, edge $e$ is not in any $3$-clique thus $L(\Gamma)$ is not $3$-clique regular.
\end{proof}
Now we restrict our focus to connected graphs and will classify all that are isomorphic to the $\omega$-clique graph of their line graph. Again, we start with the case $\omega \geq 4$ and we will find the following lemma useful.
\clearpage
\begin{lem} \label{lem:delta}
    If $\omega\geq 3$ and $\Gamma$ is a connected graph with $C_\omega(L(\Gamma)) \cong \Gamma$, then $\Delta(\Gamma) \leq \omega$.
\end{lem}
\begin{proof}
    First, assume for contradiction $\Delta(\Gamma) \geq \omega +2$ and let $v$ be a vertex in $\Gamma$ with degree $\Delta=\Delta(\Gamma)$. Observe that the number of $\omega$-cliques in the $\Delta$-clique created by $v$ in $L(\Gamma)$ is $\binom{\Delta}{\omega}$. Let $c_k$ be one of these $\omega$-cliques and we will show that the degree of $c_k$ in $C_\omega(L(\Gamma))$ is at least $\binom{\Delta}{\omega}-\binom{\Delta-\omega}{\omega}-1$. In the $\Delta$-clique created by $v$, $c_k$ is only non-adjacent with another clique if they don't share any vertices. There are $\Delta-\omega$ vertices in the $\Delta$-clique created by $v$ that are not in $c_k$ so there are $\binom{\Delta-\omega}{\omega}$ other $\omega$-cliques not adjacent to $c_k$. This implies that $c_k$ is adjacent to $\binom{\Delta}{\omega}-\binom{\Delta-\omega}{\omega}-1$ other $\omega$-cliques in the $\Delta$-clique created by $v$. So $d(c_k) \geq \binom{\Delta}{\omega}-\binom{\Delta-\omega}{\omega}-1$. By induction on $\Delta \geq \omega+2$, we will show that $\binom{\Delta}{\omega}-\binom{\Delta-\omega}{\omega}-1 > \Delta$ which implies $d(c_k) > \Delta(\Gamma)$ contradicting $C_\omega(L(\Gamma)) \cong \Gamma$. For the base case let $\Delta=\omega+2$. Since $\omega \geq 3$ then $\binom{\Delta-\omega}{\omega}=\binom{2}{\omega}=0$. So
    \begin{align*}
        \frac{\omega+1}{2} &\geq 2\\
        \frac{(\omega+2)(\omega+1)}{2} -1 & \geq 2\omega+3 >\omega+2\\
        \binom{\omega+2}{\omega}-\binom{2}{\omega}-1 &>\omega+2.
    \end{align*}
    For the inductive step, assume $\binom{\Delta}{\omega}-\binom{\Delta-\omega}{\omega}-1>\Delta$ for some $\Delta\geq \omega+2$ and we will show $\binom{\Delta+1}{\omega}-\binom{\Delta+1-\omega}{\omega}-1>\Delta+1$. Recall that for any positive integers $a$ and $b$, $\binom{a+1}{b}=\binom{a}{b}+\binom{a}{b-1}$. Then since $\binom{\Delta}{\omega-1} > \binom{\Delta-\omega}{\omega-1}$ implies $\binom{\Delta}{\omega-1} - \binom{\Delta-\omega}{\omega-1} >0$ we get,
    \[\binom{\Delta+1}{\omega}-\binom{\Delta+1-\omega}{\omega}-1 > \binom{\Delta}{\omega}-\binom{\Delta-\omega}{\omega}-1 \geq \Delta+1.\]\par
    Now for contradiction, assume $\Delta(\Gamma) = \omega +1$ and let $v$ be a vertex in $\Gamma$ with degree $\omega +1$. Then $v$ creates an $(\omega +1)$-clique in $L(\Gamma)$ and observe that every $\omega$-clique in the $(\omega+1)$-clique created by $v$ is adjacent to every other $\omega$-clique. So there is an $(\omega+1)$-clique in $C_\omega(L(\Gamma))$ since the number of $\omega$-cliques in an $(\omega+1)$-clique is $\omega+1$. Because of the assumed isomorphism, there must be an $(\omega+1)$-clique in $\Gamma$. Then since $\Delta(\Gamma)=\omega+1$ and $\Gamma$ is connected, at least one of the vertices in this $(\omega+1)$-clique has degree $\omega+1$, call this vertex $u$. Consider the set of edges in the $(\omega+1)$-clique incident with $u$. This set induces an $\omega$-clique in $L(\Gamma)$ we will call $c_k$ which we will show has degree in $C_\omega(L(\Gamma))$ greater than $\omega +1$. Since the other endpoint of every edge in $c_k$ has degree at least $\omega$, they all create at least one $\omega$-clique in $L(\Gamma)$ which is adjacent to $c_k$. Since $u$ has degree $\omega+1$ in $\Gamma$, it creates $\omega$ other $\omega$-cliques in $L(\Gamma)$ also adjacent to $c_k$. So $d(c_k) \geq 2\omega > \omega +1 = \Delta(\Gamma)$ since $\omega \geq 3$, contradicting the isomorphism.
\end{proof}
\begin{thm} \label{thm:4ci}
    Suppose $\omega \geq 4$ and $\Gamma$ is a connected graph. Then $C_\omega(L(\Gamma)) \cong \Gamma$ if and only if $\Gamma$ is $\omega$-regular.
\end{thm}
\begin{proof}
    First assume $\Gamma$ is $\omega$-regular. From Lemma \ref{lem:d(v)} there is an injective mapping from the vertices in $\Gamma$ to the $\omega$-cliques their edges induce in $L(\Gamma)$. From Remark \ref{remark} we know that the cliques created by each vertex in $\Gamma$ are the only $\omega$-cliques in $L(\Gamma)$. So the injective mapping is a bijection from the vertices in $\Gamma$ to the vertices in $C_\omega(L(\Gamma))$. We will show this bijection preserves adjacency. If vertices $v$ and $u$ are adjacent in $\Gamma$, then the cliques created by $v$ and by $u$ share the vertex $vu$ in $L(\Gamma)$. So in $C_\omega(L(\Gamma))$ these two cliques are adjacent. If $v$ and $u$ are not adjacent, then their cliques in $L(\Gamma)$ don't share a vertex so they are not adjacent in $C_\omega(L(\Gamma))$. Thus it follows $C_\omega(L(\Gamma)) \cong \Gamma$.\par
    Now assume that $\Gamma \cong C_\omega(L(\Gamma))$ which implies by Lemma \ref{lem:delta} that $\Delta(\Gamma)\leq \omega$. Let $d_i$ denote the number of vertices in $\Gamma$ with degree $i$ for $1 \leq i \leq \omega$. By Remark \ref{remark} and since $\Delta(\Gamma) \leq \omega$, we know the number of vertices in $C_\omega(L(\Gamma))$ equals $d_\omega$. We also know the number of vertices in $\Gamma$ is $\sum_{i=1}^\omega d_i$. Because these graphs are isomorphic, we have $d_\omega = \sum_{i=1}^\omega d_i$ implying $d_i = 0$ for all $i\neq \omega$. Thus $\Gamma$ is $\omega$-regular.
\end{proof}
Once again, the case of the $3$-clique graph of a line graph is much more complicated even when restricting to connected graphs.
\clearpage
\begin{thm} \label{thm:3ci}
    Suppose $\Gamma$ is a connected graph. Then $C_3(L(\Gamma)) \cong \Gamma$ if and only if:\\
           \hspace*{0.5cm} (1) the degree of every vertex in $\Gamma$ is $2$ or $3$,  \\
           \hspace*{0.5cm} (2) every degree $2$ vertex in $\Gamma$ is contained in a triangle, \\
           \hspace*{0.5cm} (3) every triangle in $\Gamma$ contains exactly one degree $2$ vertex, and \\
           \hspace*{0.5cm} (4) distinct triangles in $\Gamma$ share no vertices.\vspace{0.2cm} 
\end{thm}
The simplest example of a graph with this property is shown in Figure \ref{fig:iso_tri1} and its line graph is shown in Figure \ref{fig:iso_tri_line1}. Observe the bijective adjacency preserving mapping between the vertices of Figure \ref{fig:iso_tri1} and the triangles of Figure \ref{fig:iso_tri_line1}.
\begin{figure}[htbp]
    \begin{subfigure}{.5\textwidth}
      \centering
      \includegraphics[width=1\linewidth]{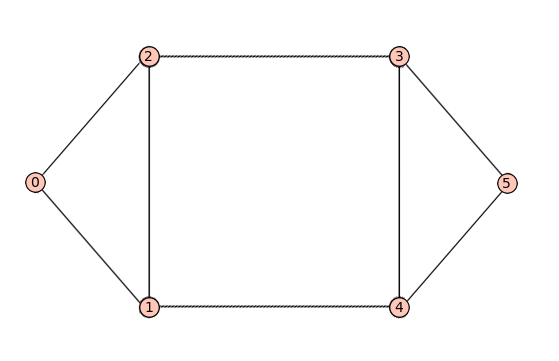}
      \caption{A graph satisfying (1), (2), (3), and (4)}
      \label{fig:iso_tri1}
    \end{subfigure}
    \begin{subfigure}{.5\textwidth}
      \centering
      \includegraphics[width=1\linewidth]{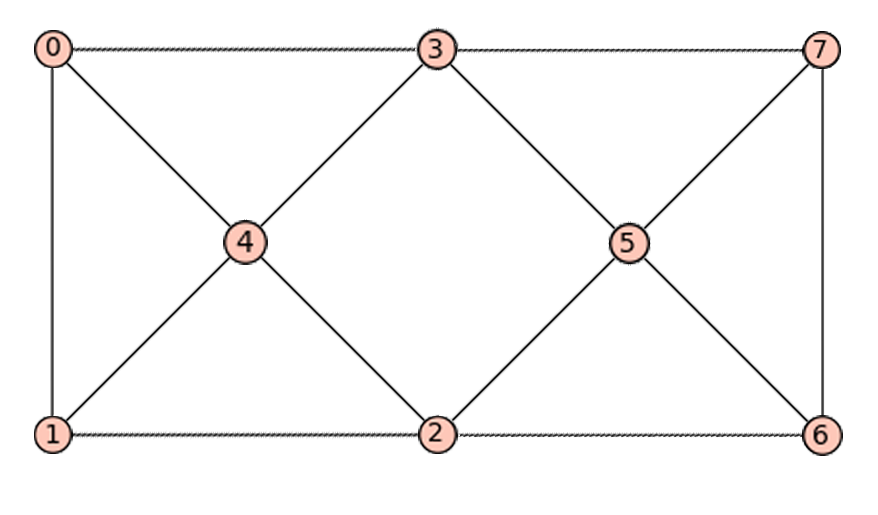}
      \caption{The line graph of (a)}
      \label{fig:iso_tri_line1}
    \end{subfigure}
    \label{fig:tri_subgraph1}
    \caption[An example of a graph satisfying the hypothesis of Theorem \ref{thm:3ci}]{}
    \end{figure}
\begin{proof}
    Begin by assuming (1), (2), (3), and (4). From these, observe that each triangle and each degree two vertex in $\Gamma$ must exist in a subgraph isomorphic to Figure \ref{fig:iso_tri}, which has line graph isomorphic to Figure \ref{fig:iso_tri_line}.
        \begin{figure}[htbp]
    \begin{subfigure}{.5\textwidth}
      \centering
      \includegraphics[width=.8\linewidth]{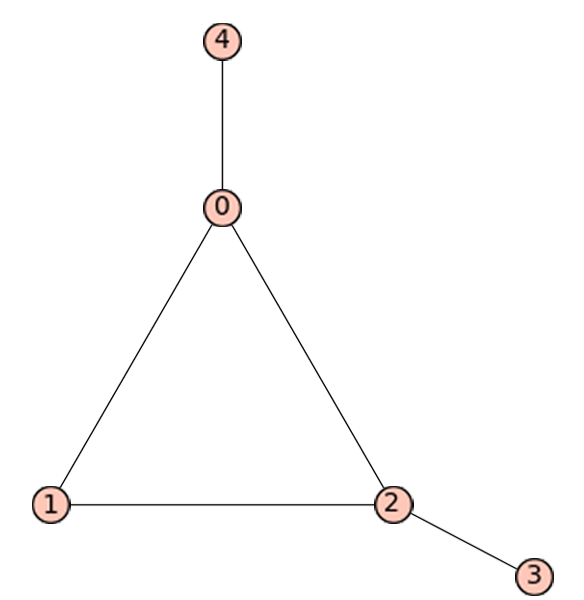}
      \caption{Subgraph of $\Gamma$}
      \label{fig:iso_tri}
    \end{subfigure}
    \begin{subfigure}{.5\textwidth}
      \centering
      \includegraphics[width=.75\linewidth]{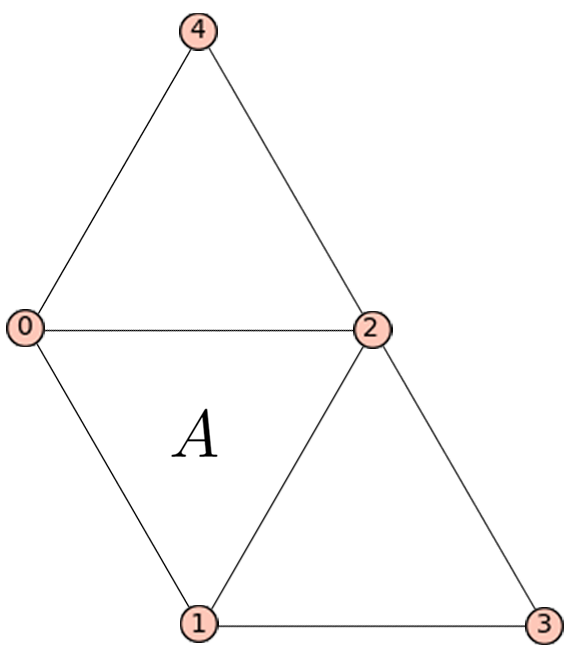}
      \caption{Subgraph of $L(\Gamma)$}
      \label{fig:iso_tri_line}
    \end{subfigure}
    \caption[Graphs used to show $C_3(L(\Gamma))\cong \Gamma$ in the proof of Theorem \ref{thm:3ci}]{}
    \end{figure}
    We define the mapping from the vertices of $\Gamma$ to the vertices of $C_3(L(\Gamma))$ by sending a degree $3$ vertex to the triangle it induces in $L(\Gamma)$ by Lemma \ref{lem:d(v)}, and sending a degree $2$ vertex to the triangle labeled $A$ in Figure \ref{fig:iso_tri_line}; the triangle in $L(\Gamma)$ induced by the unique triangle in $\Gamma$ containing the degree $2$ vertex as implied by (2). Since by Remark \ref{remark}, the only triangles in $L(\Gamma)$ are those created by a degree $3$ vertex or those induced by a triangle in $\Gamma$ with unique degree $2$ vertex, this mapping is bijective. We will now show it preserves adjacency. Let $u$ and $v$ be vertices in $\Gamma$. If $u$ and $v$ both have degree 3, the vertices they map to in $C_3(L(\Gamma))$ are adjacent if and only if $u$ and $v$ are adjacent by the same argument as in Theorem \ref{thm:4ci}. If both vertices have degree $2$, then they cannot be adjacent since by (2), this would imply they are in the same triangle, contradicting (3). And the vertices they map to in $C_3(L(\Gamma))$ are not adjacent since this would imply that two edges in the distinct triangles containing $u$ and $v$ share a vertex, contradicting (4). So now assume WLOG that $d(u)=2$ and $d(v)=3$. If $u$ and $v$ are adjacent in $\Gamma$, then Figure \ref{fig:iso_tri} is a subgraph of $\Gamma$ with $u$ as the vertex $1$ and $v$ as the vertex $0$. So $v$ is in the triangle in $\Gamma$ containing $u$ and the triangles they induce in $L(\Gamma)$ share the vertex $uv$. If the vertices that $u$ and $v$ map to in $C_3(L(\Gamma))$ are adjacent, then the triangles induced by the triangle containing $u$ and the edges incident to $v$ share some vertex in $L(\Gamma)$. This means the triangle containing $u$ in $\Gamma$ has some edge incident to $v$, implying $u$ and $v$ are adjacent. Thus $C_3(L(\Gamma)) \cong \Gamma$.\par
    For the converse, assume that $C_3(L(\Gamma)) \cong \Gamma$ which means by Lemma \ref{lem:delta} that $\Delta(\Gamma) \leq 3$. So for $1\leq i \leq 3$, let $d_i$ denote the number of vertices of degree $i$, and for $0 \leq j \leq 3$ let $t_j$ denote the number of triangles with $j$ vertices of degree $2$. From Remark \ref{remark}, the number of vertices in $C_3(L(\Gamma))$ is given by $d_3 + \sum_{j=0}^3 t_j$. First we will show $t_0 = t_3 = 0$. Since $\Gamma$ is connected, if there exists a triangle with $3$ vertices of degree $2$ then $\Gamma \cong K_3$ but $C_3(L(K_3)) \cong K_1$, a contradiction. If there exists a triangle with no vertices of degree $2$, then Figure \ref{fig:tri_t3} is a subgraph of $\Gamma$ where vertices 3, 4 and 5 may or may not be distinct. So Figure \ref{fig:tri_t3_line} is a subgraph of $L(\Gamma)$.\par
            \begin{figure}[htbp]
    \begin{subfigure}{.5\textwidth}
      \centering
      \includegraphics[width=.8\linewidth]{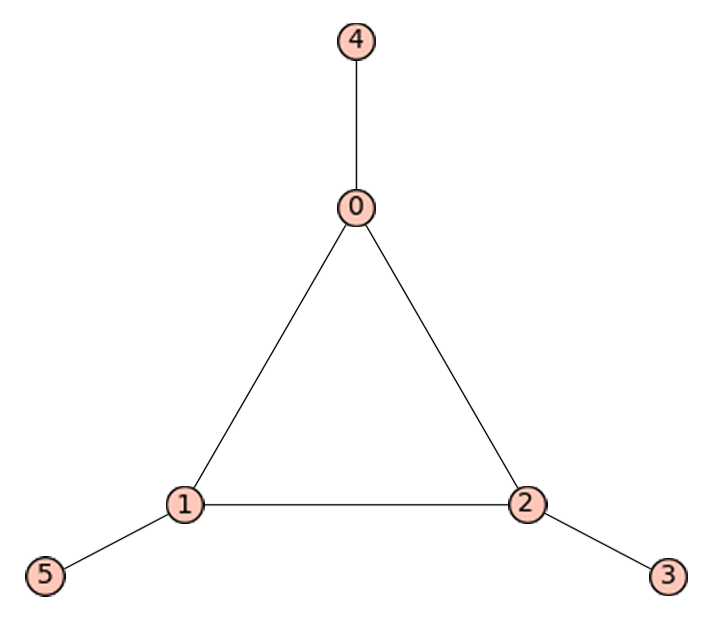}
      \caption{Subgraph of $\Gamma$}
      \label{fig:tri_t3}
    \end{subfigure}
    \begin{subfigure}{.5\textwidth}
      \centering
      \includegraphics[width=.8\linewidth]{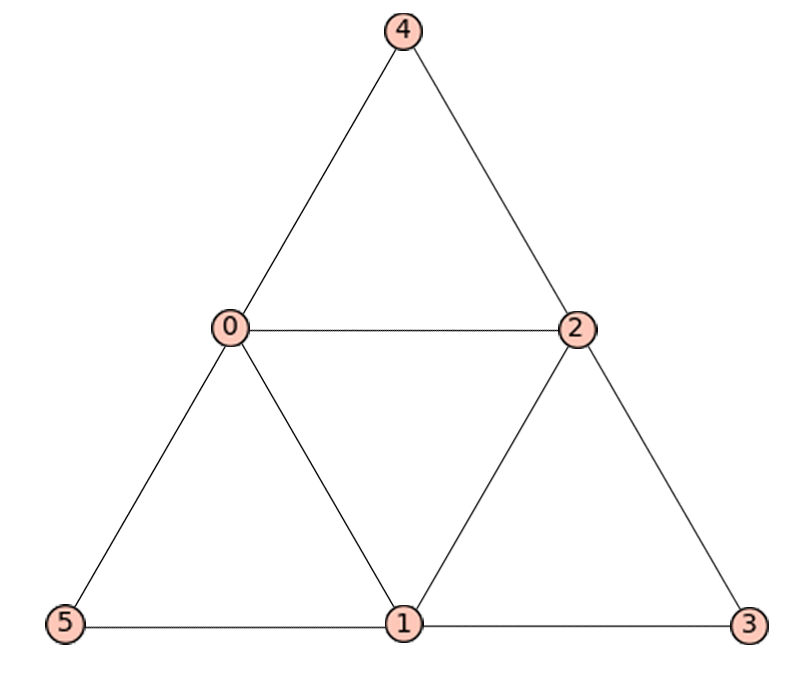}
      \caption{Subgraph of $L(\Gamma)$}
      \label{fig:tri_t3_line}
    \end{subfigure}
    \caption[Graphs used to show (1), (2), and (3) in the proof of Theorem \ref{thm:3ci}]{}
    \end{figure}
    Figure \ref{fig:tri_t3_line} consists of four triangles with each pair sharing at least one vertex, so $K_4$ is a subgraph of $C_3(L(\Gamma)) \cong \Gamma$. Since $\Gamma$ is connected and $\Delta(\Gamma) \leq 3$, we must have $\Gamma \cong K_4$. But it can be verified that $C_3(L(K_4)) \not \cong K_4$, a contradiction.\\
    So we have $t_0=t_3=0$ implying the number of vertices in $C_3(L(\Gamma))$ is $d_3 + t_1 + t_2$ and since the number of vertices in $\Gamma$ is given by $d_1 + d_2 + d_3$, their isomorphism implies $t_1 + t_2 = d_1 + d_2$. A counting argument shows that $d_2 \geq t_1 + 2t_2$ implying 
    \begin{align*}
        t_1 + t_2 &= d_1 + d_2 \\
        t_1 + t_2 &\geq d_1 + t_1 + 2t_2 \\
        -t_2 &\geq d_1.
    \end{align*}
    Thus $t_2 = d_1 = 0$, since $0 \geq -t_2 \geq d_1 \geq 0$. This implies $(1)$ and $(3)$ and since $d_2 = t_1$, we have $(2)$ as well. To show (4), observe from (1), (2), and (3) that if two distinct triangles share vertices in $\Gamma$, Figure \ref{fig:line_tri_subgraph} must be a subgraph of $\Gamma$ and so Figure \ref{fig:disjoint_tri} is a subgraph of $L(\Gamma)$.\par
        \begin{figure}[htbp]
    \begin{subfigure}{.5\textwidth}
      \centering
      \includegraphics[width=.8\linewidth]{line_tri_subgraph.png}
      \caption{Subgraph of $\Gamma$}
      \label{fig:line_tri_subgraph}
    \end{subfigure}
    \begin{subfigure}{.5\textwidth}
      \centering
      \includegraphics[width=.6\linewidth]{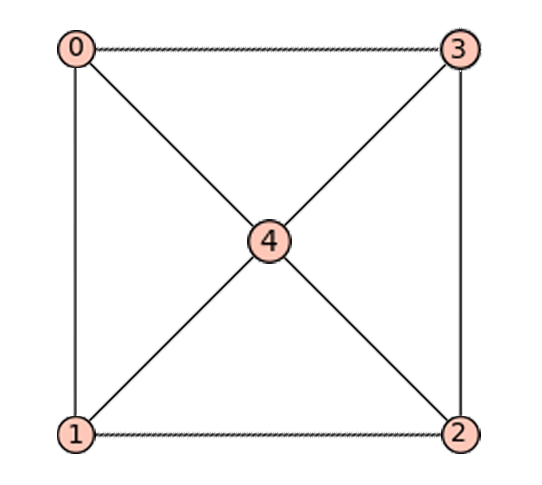}
      \caption{Subgraph of $L(\Gamma)$}
      \label{fig:disjoint_tri}
    \end{subfigure}
    \caption[Graphs used to show (4) in the proof of Theorem \ref{thm:3ci}]{}
    \end{figure}
    Figure \ref{fig:disjoint_tri} consists of four triangles all sharing vertex 4, so $K_4$ is a subgraph of $C_3(L(\Gamma))$. But as shown above, this leads to a contradiction so we get (4).
 \end{proof}

\chapter{Spectral Graph Theory}
Given a fixed ordering of vertices $V=\{v_1, v_2, \dots, v_n\}$ the \textbf{adjacency matrix} of $\Gamma$ is defined as the $n \times n$ matrix $A$ with rows and columns indexed by the vertices of $\Gamma$, with
\[(A)_{ij} = 
    \begin{cases}
        1 & \text{if $v_i$ is adjacent to $v_j$}\\
        0 & \text{otherwise}
    \end{cases}\]
where $(A)_{ij}$ is the $i, j^\text{th}$ entry of $A$. We denote by $I_n$ the $n\times n$ identity matrix. We denote the characteristic polynomial of any $n\times n$ matrix $A$ as $p(A;\lambda)=\det(\lambda I_n-A)$ and for a graph $\Gamma$, $p(\Gamma;\lambda)$ denotes the characteristic polynomial of the adjacency matrix of $\Gamma$. \par
The \textbf{spectrum} of a graph $\Gamma$ is the set of eigenvalues of the adjacency matrix of $\Gamma$ along with their multiplicities. It is denoted $\lambda_0^{a_0}, \ldots , \lambda_r^{a_r}$ where the exponent represents the multiplicity of that eigenvalue. The spectrum of a graph contains information on a number of different structural properties of the graph, which is why it is a widely studied graph invariant, and why we study the spectrum of clique graphs.\par
One such property of the spectrum is that the largest eigenvalue of a graph is always less than or equal to the maximum degree of a vertex, with equality if and only if the graph is regular \cite[Proposition 3.1.2]{haemers}. This means a $k$-regular graph has $k$ as its largest eigenvalue, and these facts will be useful later in this chapter and the next.
\section{Spectrum of a Clique Graph}
In this section, we will provide results on the spectrum of clique regular graphs and their clique graphs. We will provide bounds on their eigenvalues and in the case that $\Gamma$ is $k$-regular, an exact specification on the spectrum of $C_\omega(\Gamma)$.\par
Let $m$ be the number of $\omega$-cliques in an $\omega$-clique regular graph $\Gamma$ so that $m\binom{\omega}{2}$ is the number of edges in $\Gamma$. Let $A_C$ be the adjacency matrix of $C_\omega(\Gamma)$ with vertices indexed $(c_1, \ldots , c_{m})$ and $A_L$ be the adjacency matrix of $L(\Gamma)$ with vertices indexed $(e_1, \; \ldots, \;e_{m\binom{\omega}{2}})$ such that edges $\{e_{(i-1)\binom{\omega}{2} +1},\; \ldots, \; e_{i\binom{\omega}{2}}\}$ constitute clique $c_i$ for all $1 \leq i \leq m$, i.e. clique $c_1$ contains edges $e_1$ through $e_{\binom{\omega}{2}}$, clique $c_2$ contains edges $e_{\binom{\omega}{2}+1}$ through $e_{2\binom{\omega}{2}}$ and so on.
Now define $\tilde{A} = 6\binom{\omega}{3}I_{m} + (\omega -1)^2A_C$ and
define $\varphi : \mathbb{R}^{m} \rightarrow \mathbb{R}^{m\binom{\omega}{2}}$ by 
\[\varphi\left(\begin{bmatrix}a_1\\a_2\\\vdots\\a_{m}
\end{bmatrix}\right)=\begin{bmatrix}\Vec{a}_1\\\Vec{a}_2\\\vdots\\\Vec{a}_{m}\end{bmatrix} \text{ where } \:\ \Vec{a}:=
    \begin{bmatrix}a\\\vdots\\a\end{bmatrix} \in \mathbb{R}^{\binom{\omega}{2}}.\]
So the $\varphi$ function expands a column vector by a factor of $\binom{\omega}{2}$. For example, if $m=2$ and $\omega = 3$ then $\varphi\left(\begin{bmatrix}1\\2\end{bmatrix}\right)=\begin{bmatrix}1 &1&1&2&2&2\end{bmatrix}^\top$.
\begin{lem} \label{lem:phi}
If $\Gamma$ is $\omega$-clique regular with $\tilde{A},\; \varphi$, and $A_L$ as above, then for all $v \in \mathbb{R}^{m}$, \[v^\top \tilde{A}v = \varphi(v)^\top A_L \varphi(v).\]
\end{lem}
\begin{proof}
    Let $t_i$ be the standard basis vector of $\mathbb{R}^{m}$ with a $1$  in the $i^{\text{th}}$ position for $1 \leq i \leq m$.
    It is sufficient to show that for all $t_i$ and $t_j$
    \[t_i^\top  \tilde{A}t_j = \varphi(t_i)^\top  A_L \varphi(t_j).\]
    Let $S_{ij}$ denote the sub-matrix of $A_L$ with rows corresponding to edges in $c_i$ and columns corresponding to edges in $c_j$. Observe that $\varphi(t_i)^\top  A_L \varphi(t_j)$ equals the sum of the entries of $S_{ij}$. Similar to above, it is clear that $t_i^\top  \tilde{A}t_j$ is the $1\times 1$ matrix with entry $(\tilde{A})_{ij}$.\par
    Now for the first case $i=j$. Let $e_k$ be an edge in $\Gamma$ with endpoints $x$ and $y$ in $c_i$. So $e_k$ corresponds to a row in $S_{ii}$. The vertex $x$ is incident to $\omega -1$ edges in $c_i$ one of which is $e_k$ and the same for vertex $y$. So $e_k$ is incident to $2(\omega-2)$ edges in $c_i$ implying the sum of the $e_k$ row in $S_{ii}$ is $2(\omega-2)$. Since there are $\binom{\omega}{2}$ rows in $S_{ii}$, the total sum of entries in $S_{ii}$ is $2(\omega-2)\binom{\omega}{2} = 6\binom{\omega}{3}$. So $t_i^\top  \tilde{A}t_i =\left[6\binom{\omega}{3}\right] = \varphi(t_i)^\top  A_L \varphi(t_i)$.\par
    Next for the case when $c_i$ is adjacent to $c_j$. Let $x$ be the unique vertex in $\Gamma$ that $c_i$ and $c_j$ share and let $e_k$ be an edge incident to $x$ with other endpoint in $c_i$. Then $e_k$ is incident to the $\omega-1$ edges of $x$ with other endpoints in $c_j$ implying the sum of the $e_k$ row in $S_{ij}$ is $\omega-1$. Since there are $\omega-1$ edges incident to $x$ with other endpoint in $c_i$, and since every edge in $c_i$ that isn't incident to $x$ is not incident to any edges in $c_j$, the total sum of the entries in $S_{ij}$ is $(\omega-1)^2$. So $t_i^\top  \tilde{A}t_j =\left[(\omega-1)^2\right] = \varphi(t_i)^\top  A_L \varphi(t_j)$.\par
    Finally for the last case when $i \neq j$ and $c_i$ is not adjacent to $c_j$. Let $e_k$ be an edge in $c_i$. Then since $c_i$ and $c_j$ don't share common vertices, $e_k$ is not incident to any edges in $c_j$ implying that the total sum of the entries in $S_{ij}$ is 0.\\So $t_i^\top  \tilde{A}t_j =\left[0\right] = \varphi(t_i)^\top  A_L \varphi(t_j)$.
\end{proof}

\begin{thm} \label{thm:bounds}
    Suppose $\Gamma$ is $\omega$-clique regular and the eigenvalues of $L(\Gamma)$ are $\mu_1 \leq \cdots \leq \mu_{m\binom{\omega}{2}}$. Then for each eigenvalue $\lambda$ of $C_\omega(\Gamma)$,
    \[\frac{\omega}{\omega - 1}\left(\frac{\mu_1}{2} -\omega +2\right) \leq \lambda \leq \frac{\omega}{\omega - 1}\left(\frac{\mu_{m\binom{\omega}{2}}}{2} -\omega +2\right).\]
\end{thm}
\begin{proof}
    For this proof $\|t\|$ will denote the euclidean norm of vector $t$. Let $\lambda$ be an eigenvalue of $A_C$ with eigenvector $u$. Clearly, $A_C$ and $\tilde{A} = 6\binom{\omega}{3}I_{m} + (\omega -1)^2A_C$ share the same eigenvectors and $\eta= 6\binom{\omega}{3}+(\omega-1)^2\lambda$ is an eigenvalue of $\tilde{A}$ with the eigenvector $u$. Now define $v = \frac{u}{\|u\|\sqrt{\binom{\omega}{2}}}$ so that $\|v\|^2=\frac{1}{\binom{\omega}{2}}$ and $v$ is an eigenvector of $\tilde{A}$ with eigenvalue $\eta$. Notice that 
    \[\varphi(v)^\top \varphi(v)=\left[|\varphi(v)|^2\right]=\left[\binom{\omega}{2}\|v\|^2\right]=[1]=I_1.\]
    By the Eigenvalue Interlacing Theorem, \cite[pg. 26]{haemers} this implies that the eigenvalue of the $1 \times 1$ matrix $\varphi(v)^\top  A_L \varphi(v)$ interlaces the eigenvalues of $A_L$. By Lemma \ref{lem:phi}, 
    \[\varphi(v)^\top  A_L \varphi(v)  = v^\top  \tilde{A}v = \eta v^\top  v =\left[\eta\|v\|^2\right]=\left[\frac{\eta}{\binom{\omega}{2}}\right].\] So the eigenvalue $\frac{\eta}{\binom{\omega}{2}}$ interlaces the eigenvalues of $A_L$, $\mu_1 \leq \frac{\eta}{\binom{\omega}{2}} \leq \mu_{m\binom{\omega}{2}}$. \\Since $\eta=6\binom{\omega}{3}+(\omega - 1)^2\lambda$, it follows that
     \[\frac{1}{(\omega - 1)^2}\left(\binom{\omega}{2} \mu_1-6\binom{\omega}{3}\right) \leq \lambda \leq \frac{1}{(\omega - 1)^2}\left(\binom{\omega}{2} \mu_{m\binom{\omega}{2}}-6\binom{\omega}{3}\right).\]
    Equivalently,
    \[\frac{\omega}{\omega - 1}\left(\frac{\mu_1}{2} -\omega +2\right) \leq \lambda \leq \frac{\omega}{\omega - 1}\left(\frac{\mu_{m\binom{\omega}{2}}}{2} -\omega +2\right).\]
\end{proof}
From this theorem, we can also derive bounds on the eigenvalues that don't require any knowledge of the line graph.
\begin{cor} \label{cor:eigens}
    If $\Gamma$ is $\omega$-clique regular then for each eigenvalue $\lambda$ of $C_\omega(\Gamma)$, \[-\omega \leq \lambda \leq \omega\left(\frac{\Delta(\Gamma)}{\omega-1}-1\right)\] where $\Delta(\Gamma)$ denotes the largest degree of $\Gamma$.
\end{cor} 
\begin{proof}
    Let $\mu_1 \leq \cdots \leq \mu_{m\binom{\omega}{2}}$ be the eigenvalues of $L(\Gamma)$. Two classical results in spectral graph theory are that for any line graph, $-2 \leq \mu_1$ and $\mu_{m\binom{\omega}{2}} \leq \Delta(L(\Gamma))\leq 2(\Delta(\Gamma)-1)$. From these we get 
    \[ -\omega \leq \frac{\omega}{\omega - 1}\left(\frac{\mu_1}{2} -\omega +2\right)\] and \[ \frac{\omega}{\omega - 1}\left(\frac{\mu_{m\binom{\omega}{2}}}{2} -\omega +2\right) \leq \omega\left(\frac{\Delta(\Gamma)}{\omega-1}-1\right).\]
    So the result follows from Theorem \ref{thm:bounds}.
\end{proof}\par
If $\Gamma$ is an $\omega$-clique regular graph and also $k$-regular, then we can show that the characteristic polynomial of $C_\omega(\Gamma)$ is a function of the characteristic polynomial of $\Gamma$. Recall that $p(\Gamma; \lambda)$ denotes the characteristic polynomial of $\Gamma$'s adjacency matrix and that the roots of this polynomial are the eigenvalues of $\Gamma$. Let $n$ be the number of vertices in $\Gamma$ and $m=\frac{nk}{\omega (\omega -1)}$ be the number of $\omega$-cliques in $\Gamma$, i.e. the number of vertices in $C_\omega(\Gamma)$.\par
Let $A$ be the adjacency matrix of $\Gamma$ with vertices indexed $(v_1, \ldots , v_n)$, and let $A_C$ be the adjacency matrix of $C_\omega(\Gamma)$ with $\omega$-cliques indexed $(c_1, \ldots , c_m)$. We will define the \textbf{$\omega$-clique incidence matrix} of $\Gamma$ as an $n\times m$ matrix $R$ such that the rows are indexed by the vertices in $\Gamma$ in the same order as in $A$, and the columns are indexed by the $\omega$-cliques in $\Gamma$ in the same order as in $A_C$. The entries of $R$ are defined,
\[(R)_{ij} = 
    \begin{cases}
        1 & \text{if } v_i \in c_j\\
        0 & \text{otherwise.}
    \end{cases}\]
Recall also that the \textbf{degree matrix} of $\Gamma$ is defined as $D=\text{diag}(d(v_1), \ldots , d(v_n))$.

\begin{lem} \label{lem:blockmat}
Suppose that $\Gamma$ is $\omega$-clique regular with $\omega$-clique incidence matrix $R$ and degree matrix $D$. Then:\vspace{0.2cm} \\
           \hspace*{0.5cm} (1) $R^\top R=A_C + \omega I_m$ and  \\
           \hspace*{0.5cm} (2) $RR^\top  = A + \frac{1}{\omega -1}D$. \vspace{0.2cm}
\end{lem}
\begin{proof}
    $(1)$ We have
    \[(R^\top R)_{ij} = \sum_{l=1}^n (R)_{li}(R)_{lj},\] 
    from which it is clear that $(R^\top R)_{ij}$ is the number of vertices that are in both clique $c_i$ and $c_j$. If $i=j$, then the sum will equal $\omega$, which is the number of vertices in $c_i$. So $(R^\top R)_{ii} = \omega =(A_C + \omega I_m)_{ii}$. If $c_i$ and $c_j$ are adjacent, then they share one vertex in common. So $(R^\top R)_{ij} = 1 = (A_C + \omega I_m)_{ij}$. Otherwise $(R^\top R)_{ij}=0=(A_C + \omega I_m)_{ij}$.\\
    $(2)$ Similarly, we have that $(RR^\top )_{ij}$ is the number of cliques that $v_i$ and $v_j$ commonly belong to. If $i=j$, then the sum will equal the number of cliques that contain $v_i$. Since $v_i$ has degree $d(v_i)$, and each clique with $v_i$ contains $\omega -1$ edges adjacent to $v_i$ unique to that clique, $\frac{d(v_i)}{\omega -1}$ is the number of cliques that $v_i$ belongs to. So $(RR^\top )_{ij} = \frac{d(v_i)}{\omega -1} = (A + \frac{1}{\omega -1}D)_{ij}$. If $v_i$ and $v_j$ are adjacent, then the edge between them belongs to a unique clique so $(RR^\top )_{ij} = 1 = (A + \frac{1}{\omega -1}D)_{ij}$. Otherwise $(RR^\top )_{ij} = 0 =(A + \frac{1}{\omega -1}D)_{ij}$.
\end{proof}
\begin{thm} \label{thm:eigen}
    If $\Gamma$ is $k$-regular and $\omega$-clique regular, then 
    \[p(C_\omega(\Gamma); \lambda)=(\lambda+\omega)^{m-n}p\left(\Gamma;\lambda+\omega-\frac{k}{\omega - 1}\right)\]
    where $m=\frac{nk}{\omega(\omega-1)}$.
\end{thm}
\begin{proof}
    Define two square block matrices with $n+m$ rows and columns as follows with $R$ the $\omega$-clique incidence matrix of $\Gamma$,
    \[U = \begin{bmatrix}(\lambda+\omega) I_n & -R \\0 & I_m \\\end{bmatrix}, \hspace{1cm}
    V = \begin{bmatrix}I_n & R \\R^\top  & (\lambda+\omega) I_m \\\end{bmatrix}.\]
Then, 
\[UV = \begin{bmatrix}(\lambda+\omega) I_n - RR^\top  & 0 \\R^\top  & (\lambda+\omega) I_m \\\end{bmatrix}, 
    VU = \begin{bmatrix}(\lambda+\omega) I_n & 0 \\(\lambda+\omega) R^\top  & (\lambda+\omega) I_m - R^\top R \\\end{bmatrix}.\]
Since $\text{det}(UV)=\text{det}(VU)$, it follows that
\begin{align*}
    (\lambda+\omega)^m \text{det}((\lambda+\omega) I_n - RR^\top ) & =(\lambda+\omega)^n \text{det}((\lambda+\omega) I_m - R^\top R) \\ 
    (\lambda+\omega)^{m-n} \text{det}((\lambda+\omega) I_n - RR^\top ) & =\text{det}((\lambda+\omega) I_m - R^\top R).
\end{align*}
Using the above equality and Lemma \ref{lem:blockmat} we get,
\begin{align*}
    p(C_\omega(\Gamma);\lambda) & = \text{det}(\lambda I_m - A_C) \\ 
    & = \text{det}((\lambda + \omega)I_m - R^\top R)  \\
    & = (\lambda + \omega)^{m-n}\text{det}((\lambda + \omega)I_n - RR^\top ) \\
    & = (\lambda + \omega)^{m-n}\text{det}\left(\left(\lambda + \omega - \frac{k}{\omega -1}\right)I_n - A\right)\\
    & = (\lambda + \omega)^{m-n}p\left(\Gamma ; \lambda + \omega - \frac{k}{\omega -1}\right).
\end{align*}
\end{proof}
Note that for the previous lemma and theorem, taking $\omega =2$ gives the proof from Biggs \cite[p.18-19]{biggs} for the characteristic polynomial of the line graph of a regular graph. Here we generalized this proof for the $\omega$-clique graph of any graph that is regular and $\omega$-clique regular for any $\omega \geq 2$.\par
\begin{remark} \label{rem:eigen}
    Recall that the eigenvalues of a matrix are exactly the roots of the matrix's characteristic polynomial. So Theorem \ref{thm:eigen} implies that if the spectrum of a $k$-regular and $\omega$-clique regular graph is
    \[k^{a_0},\lambda_1^{a_1}, \ldots , \lambda_r^{a_r},\]
    then the spectrum of its $\omega$-clique graph is 
    \[\left(\frac{k}{\omega -1} + k -\omega \right)^{a_0}, \left(\frac{k}{\omega -1} + \lambda_1 -\omega \right)^{a_1}, \ldots, \left(\frac{k}{\omega -1} + \lambda_r -\omega \right)^{a_r}, -\omega^{m-n}\]
    where $n$ is the number of vertices and $m=\frac{nk}{\omega(\omega -1)}$ is the number of $\omega$-cliques.
\end{remark}
\clearpage
\begin{cor} \label{cor:lowbound}
    Suppose $\Gamma$ is $k$-regular and $\omega$-clique regular where the smallest eigenvalue $\lambda_1$ has multiplicity $a_1$. Then $\lambda_1 \geq \frac{-k}{\omega-1}$ and if $k < \omega(\omega -1)$ as well, then $\lambda_1 = \frac{-k}{\omega-1}$ and $a_1 \geq n-\frac{nk}{\omega(\omega -1)}.$
\end{cor}
\begin{proof}
    From Remark \ref{rem:eigen}, $\frac{k}{\omega -1}+\lambda_1 -\omega$ is an eigenvalue of $C_\omega(\Gamma)$ so from Corollary \ref{cor:eigens} $-\omega \leq \frac{k}{\omega -1} +\lambda_1 -\omega$ implies $\lambda_1 \geq \frac{-k}{\omega -1}$. If $k < \omega(\omega-1)$ then $m-n=\frac{nk}{\omega(\omega -1)}-n <0$. Since this is the exponent of the $(\lambda + \omega)$ term in the characteristic polynomial of $C_\omega(\Gamma)$ by Theorem \ref{thm:eigen}, there must be the same term in $p(\Gamma; \lambda +\omega - \frac{k}{\omega -1})$ with exponent greater than or equal to $n-\frac{nk}{\omega(\omega -1)}$. By above, this term must come from the smallest eigenvalue so we get $\lambda_1 = \frac{-k}{\omega -1}$ with multiplicity $a_1 \geq n - \frac{nk}{\omega(\omega -1)}$.
\end{proof}
\chapter{Strongly Regular Graphs}
A graph $\Gamma$ is called \textbf{strongly regular} with parameters ($n,k,\lambda,\mu$) and is abbreviated as srg($n,k,\lambda,\mu$) if
\begin{itemize}
    \item $\Gamma$ has $n$ vertices,
    \item $\Gamma$ is $k$-regular,
    \item Adjacent vertices share exactly $\lambda$ common neighbors, and
    \item Distinct non-adjacent vertices share exactly $\mu$ common neighbors.
\end{itemize}
Clearly, strongly regular graphs are a subset of edge regular graphs as defined in the introduction, but with an additional requirement on non-adjacent vertices.
 The parameters ($n,k,\lambda,\mu$) are closely related and must obey, 
\begin{equation}\label{eq:srg}(n-k-1)\mu=k(k-\lambda-1),\end{equation}
 \cite[Theorem 9.1.3]{haemers}.\par
 A graph $\Gamma$ is strongly regular if and only if it is regular with 3 or fewer distinct eigenvalues \cite[Theorem 9.1.2]{haemers}. So we can classify strongly regular graphs as either ``boring'' (graphs with one or two eigenvalues), disjoint unions of complete graphs and their complements, or ``non-boring'' if otherwise. Researchers focus mainly on non-boring strongly regular graphs, which we will do as well, and a lot of research is devoted to studying their spectral properties. The spectrum of a strongly regular graph is derivable from its parameters, which we record in the following theorem.
 \clearpage
 \begin{thm}\label{thm:srgspec}\cite[Theorem 9.1.2]{haemers} If $\Gamma$ is a non-boring srg$(n,k,\lambda, \mu)$, then it will have spectrum $k^1, r^f, s^g$ where $k>r>s$ and the latter two eigenvalues and their multiplicities are given by the following,
\begin{equation}\label{eq:srgspec}\begin{split}r = \frac{1}{2}\left[(\lambda - \mu) + \sqrt{(\lambda - \mu)^2+4(k-\mu)}\right],\\
f = \frac{1}{2}\left[(n-1)-\frac{2k + (n-1)(\lambda-\mu)}{\sqrt{(\lambda - \mu)^2+4(k-\mu)}}\right],\\
s = \frac{1}{2}\left[(\lambda - \mu) - \sqrt{(\lambda - \mu)^2+4(k-\mu)}\right],\\
g = \frac{1}{2}\left[(n-1)+\frac{2k + (n-1)(\lambda-\mu)}{\sqrt{(\lambda - \mu)^2+4(k-\mu)}}\right].\end{split}\end{equation}\\
From these it is also provable that $\lambda -\mu = r+s$ and $k-\mu = -rs$.\end{thm}
 One of the biggest questions on the topic of strongly regular graphs is: for which sets of parameters ($n,k,\lambda,\mu$) does a strongly regular graph exist? Recall from the introduction the statement of Conway's 99-graph problem: does there exist a graph on 99 vertices that is 14-regular such that every pair of adjacent vertices is contained in a unique triangle and every pair of non-adjacent vertices is contained in a unique quadrangle? We see now that this is equivalent to asking if there exists an srg$(99,14,1,2)$.\par
 We call a parameter set $(n,k,\lambda,\mu)$ feasible if it satisfies equation (\ref{eq:srg}) and its implied spectrum given by (\ref{eq:srgspec}) are integers with positive integer multiplicities. A large list of feasible parameter sets can be found on Andries Brouwer's website \cite{brouwer}. Brouwer's list categorizes each set of srg parameters into three groups (existence, non-existence, and unknown). Boring graphs are excluded from this list.\par
 \section{Strongly Regular Clique Regular Graphs}\label{sec:srgcrg}
 There is much overlap between clique regular graphs and strongly regular graphs. If an srg$(n,k,\lambda,\mu)$ $\Gamma$ is $\omega$-clique regular, since we know its spectrum is given by $k^1,r^f,s^g$ from Theorem \ref{thm:srgspec}, Remark \ref{rem:eigen} implies that the spectrum of $C_\omega(\Gamma)$ is given by \begin{equation}\label{eq:spec}\left(\frac{k}{\omega-1}+k-\omega \right)^1,\;\left(\frac{k}{\omega-1}+r-\omega\right)^f,\;  \left(\frac{k}{\omega-1}+s-\omega\right)^g, \;-\omega^{m-n} \end{equation} where $m=\frac{nk}{\omega(\omega-1)}$.  The following theorem will determine when the $\omega$-clique graph of an $\omega$-clique regular and strongly regular graph is also strongly regular. 
\begin{thm} \label{thm:srg}
    Suppose $\Gamma$ is $\omega$-clique regular and a non-boring \textnormal{srg}$(n,k,\lambda, \mu)$ with spectrum $k^1,r^f,s^g$ where $r>s$. Then the $\omega$-clique graph of $\Gamma$ is strongly regular if and only if $s = \frac{-k}{\omega-1}$ or $k = \omega(\omega -1)$. If so, $C_\omega(\Gamma)$ has parameters
    $$\textnormal{srg}\left( \frac{nk}{\omega(\omega -1)},\; \omega \left(\frac{k}{\omega-1}-1\right),\; \lambda^*,\; \mu^* \right)$$
    and if $\lambda = \omega - 2$, then $\lambda^* = \frac{k}{\omega -1}-2$ and $\mu^*=\mu +\omega-\frac{k}{\omega-1}$.
\end{thm}
Note that $\lambda^*$ and $\mu^*$ are forced by the spectrum of $C_\omega(\Gamma)$ and can be always be derived from it, but if $\Gamma$ is a regular clique assembly (equivalent to $\lambda = \omega - 2$ by Theorem \ref{thm:rca1}) we can derive simpler formulas for the srg parameters.
\begin{proof}
    Since $C_\omega(\Gamma)$ is $\omega(\frac{k}{\omega-1}-1)$-regular, then it is sufficient to show that the adjacency matrix of  $C_\omega(\Gamma)$ has 3 or less distinct eigenvalues if and only if $s = \frac{-k}{\omega-1}$ or $k = \omega(\omega -1)$. The spectrum of $C_3(\Gamma)$ is given by (\ref{eq:spec}). If $s = \frac{-k}{\omega-1}$, then $\left(\frac{k}{\omega-1}+s-\omega\right) =-\omega$ and if $k=\omega(\omega-1)$ then $m=n$  implies the multiplicity of $-\omega$ is 0. So regardless, $C_\omega(\Gamma)$ has no more than 3 distinct eigenvalues.\par
    Conversely, assume $C_\omega(\Gamma)$ has 3 or less distinct eigenvalues. This implies that one of the multiplicities in the spectrum (\ref{eq:spec}) equals 0 or two of the eigenvalues are equal. If the first case, then $m-n$ must equal 0 since $1,f,g \neq 0$. Then $m=\frac{nk}{\omega(\omega-1)}=n$ implies $k=\omega(\omega-1)$. If two of the eigenvalues of are equal, we will show $s=\frac{-k}{\omega-1}$. If any of the latter 3 eigenvalues in (\ref{eq:spec}) are equal, it would imply that $k=r$, $k=s$ or $r=s$, a contradiction. So $-\omega$ is equal to one of the other eigenvalues implying $k, r$ or $s$ must be equal to $\frac{-k}{\omega -1}$. From Corollary \ref{cor:lowbound}, we have $k>r>s \geq \frac{-k}{\omega - 1}$ implying $s=\frac{-k}{\omega -1}$.\par
    Now assume that $C_\omega(\Gamma)$ is strongly regular and $\lambda = \omega -2$. For the case $\omega=2$ or $3$, the result $\lambda^*=\frac{k}{\omega-1}-2$ is shown in \cite[pg. 304]{guest}. So assume $\omega \geq 4$.\\
    Let $c_1$ and $c_2$ be adjacent vertices in $C_\omega(\Gamma)$ (that is,  $\omega$-cliques in $\Gamma$) and let $x \in \Gamma$ be the unique vertex that $c_1$ and $c_2$ share. Then $c_1$ and $c_2$ are commonly adjacent to all the other cliques containing $x$. So $\lambda^* \geq \frac{k}{\omega-1} -2$. Suppose for contradiction there exists another clique $c_i$ commonly adjacent to $c_1$ and $c_2$ such that $x$ is not in $c_i$. Let $y_1$ and $y_2$ be the unique common vertices of $c_1$ with $c_i$, and $c_2$ with $c_i$ respectively. Since $y_1$ and $y_2$ are both in clique $c_i$, they share $\omega -2$ common neighbors in that clique. And because $\lambda=\omega -2$, these are all the common neighbors that they share which contradicts the fact that $x$ is adjacent to both and $x$ is not in $c_i$.
    So $\lambda^* = \frac{k}{\omega -1} -2$.\\
    For $\mu^*$ recall that for any srg, $\lambda - \mu = r+s$. So we get
    \begin{align*}
        \lambda^* - \mu^* &= r^*+s^* \\
        \mu^* &= \lambda^* -r^*-s^* \\
        &= \left(\frac{k}{\omega-1}-2\right)-\left(\frac{k}{\omega -1} +r -\omega\right) -\left(\frac{k}{\omega -1} +s -\omega\right) \\
        &= 2\omega - (r+s)-2-\frac{k}{\omega-1} \\
        &= 2\omega - (\lambda - \mu)-2-\frac{k}{\omega-1}\textbf{}\\
        &= 2\omega - (\omega-2 - \mu)-2-\frac{k}{\omega-1}\\
        &= \mu +\omega- \frac{k}{\omega-1}.
    \end{align*}
\end{proof}
\clearpage
\begin{cor} \label{cor:kww}
    Suppose $\Gamma$ is $\omega$-clique regular and a non-boring \textnormal{srg}$(n,k,\lambda, \mu)$. Then $\Gamma$ and $C_\omega(\Gamma)$ are srgs with the same parameters if and only if $k=\omega(\omega-1)$.
\end{cor}
\begin{proof}
    If $k=\omega(\omega-1)$, then it follows that $\Gamma$ and $C_\omega(\Gamma)$ have the same spectrum, and it is easily provable that strongly regular graphs with the same spectrum have the same parameters. Conversely if $\Gamma$ and $C_\omega(\Gamma)$ have the same parameters then $\frac{nk}{\omega(\omega -1)}=n$ which implies $k=\omega(\omega -1)$.
\end{proof}\par
Even if a strongly regular clique regular graph $\Gamma$ does not have a strongly regular clique graph, $C_\omega(\Gamma)$ would still have structural properties that we can derive from its spectrum. Since it would have exactly the four eigenvalues given in (\ref{eq:spec}) and would be regular, it would also have another property called walk-regularity \cite[Corollary 15.1.2]{haemers}. This means that for each vertex in the graph, the number of closed walks of a particular length is the same and is given by the sum of all eigenvalues of the graph raised to the power $\ell$ and divided by the number of vertices of the graph. In our case for example, the number of closed walks at any vertex $v\in C_\omega(\Gamma)$ of length $\ell\geq 1$ is given by the expression
\begin{equation} \label{eq:walk}
\frac{1}{m}\left(\left(\omega\left(\frac{k}{\omega-1}-1\right)\right)^\ell+f\left(\frac{k}{\omega -1} +r - \omega\right)^\ell+g\left(\frac{k}{\omega -1} +s - \omega\right)^\ell+(m-n)(-\omega)^\ell\right)\end{equation}
where $m=\frac{nk}{\omega(\omega-1)}$.\par
This gives a necessary condition on the existence of such a graph, since the expression (\ref{eq:walk}) must be a non-negative integer for every positive integer $\ell$. Simply put, if we can find a feasible parameter set $(n,k,\lambda,\mu)$ for which it is provable that any strongly regular graph with these parameters must be $\omega$-clique regular, but there exists a positive integer $\ell$ such that the expression (\ref{eq:walk}) is not a positive integer, it would show there exists no strongly regular graph with parameters $(n,k,\lambda,\mu)$.\par
If we were to make use of this necessary condition, there would then be two challenges: finding parameter sets for which any strongly regular graph must be clique regular, and finding an integer $\ell$ for which the expression (\ref{eq:walk}) is negative or not an integer. We deal with the first challenge in the next section \ref{sec:localsrgs}, in which we discuss a family of feasible parameter sets for which graphs must be clique regular. Unfortunately however for the second challenge, we can show that it is functionally impossible.\par
 The first two exponentials in expression \ref{eq:walk} are always positive and almost always larger in magnitude than the last two exponentials, so the whole expression is always positive for any $\ell\geq 1$ as long as $(n,k,\lambda,\mu)$ is a feasible parameter set. We won't show this claim rigorously since it is semi-intuitive and not very important for the remainder of our discussion. The important part is we can show that as long as $(n,k,\lambda,\mu)$ is a feasible parameter set, the more restrictive condition holds; expression (\ref{eq:walk}) is an integer for all $\ell \geq 1$.
\begin{prop}
    Suppose $(n,k,\lambda, \mu)$ is a feasible strongly regular parameter set with spectrum $k^1,r^f, s^g$ and $\omega \geq 2$ is such that $m=\frac{nk}{\omega(\omega -1)}$ and $\frac{k}{\omega -1}$ are integers. Then $m$ divides 
    \[\left(\omega\left(\frac{k}{\omega-1}-1\right)\right)^\ell+f\left(\frac{k}{\omega -1} +r - \omega\right)^\ell+g\left(\frac{k}{\omega -1} +s - \omega\right)^\ell+(m-n)(-\omega)^\ell\]
    for all non-negative integers $\ell$.
\end{prop}
Note that we only need to show this for positive integers $\ell$, but since we will use stacked inductions on $\ell$, having the base case $\ell=0$ makes for an easier proof.
\begin{proof}
    For brevity, define $\tilde{k}=\omega\left(\frac{k}{\omega -1} -1\right)$, $\tilde{r}=\left(\frac{k}{\omega -1} +r - \omega\right)$ and $\tilde{s}=\left(\frac{k}{\omega -1} +s - \omega\right)$, and let ``$\equiv$'' represent equivalence modulo $m$. We will show the theorem by induction on $\ell \geq 0$. The algebra in each base case follows from equations (\ref{eq:srg}) and (\ref{eq:srgspec}). For the first base case $\ell=0$ we get $\tilde{k}^0 + f\tilde{r}^0 + g\tilde{s}^0 + (m-n)(-\omega)^0 = 1+f+g+m-n=m$ since $1+f+g=n$. Assuming $\tilde{k}^\ell+f\tilde{r}^\ell+g\tilde{s}^\ell-n(-\omega)^\ell \equiv 0$ for some $\ell \geq 0$ gives \begin{align*}
        &\tilde{k}^{\ell +1} + f\tilde{r}^{\ell +1} + g\tilde{s}^{\ell +1} -n(-\omega)^{\ell +1}\\  &\equiv \tilde{k}^\ell\left(\omega\left(\frac{k}{\omega -1} -1\right)\right) + f\tilde{r}^\ell\left(\frac{k}{\omega -1} +r - \omega\right) + g\tilde{s}^\ell\left(\frac{k}{\omega -1} +s - \omega\right) + \omega(n(-\omega)^\ell)\\  &\equiv 
        -\omega(\tilde{k}^\ell+f\tilde{r}^\ell+g\tilde{s}^\ell-n(-\omega)^\ell)+ \frac{k}{\omega -1} (\tilde{k}^\ell+f\tilde{r}^\ell+g\tilde{s}^\ell)+k\tilde{k}^\ell+fr\tilde{r}^\ell+gs\tilde{s}^\ell\\ &\equiv 
        \frac{k}{\omega -1}(n(-\omega)^\ell)+k\tilde{k}^\ell+fr\tilde{r}^\ell+gs\tilde{s}^\ell\\ &\equiv 
        \omega m(-\omega)^\ell+k\tilde{k}^\ell+fr\tilde{r}^\ell+gs\tilde{s}^\ell\\ &\equiv 
        k\tilde{k}^\ell+fr\tilde{r}^\ell+gs\tilde{s}^\ell
    \end{align*}
    So it suffices to show that $m$ divides $k\tilde{k}^\ell+fr\tilde{r}^\ell+gs\tilde{s}^\ell$ which we will do by induction on $\ell \geq 0$. The base case $\ell=0$ is $k+fr+gs=0$. Assume $k\tilde{k}^\ell+fr\tilde{r}^\ell+gs\tilde{s}^\ell \equiv 0$ for some $\ell \geq 0$ and we get 
    \begin{align*}
        & k\tilde{k}^{\ell +1}+fr\tilde{r}^{\ell +1}+gs\tilde{s}^{\ell +1} \\ & \equiv k\tilde{k}^\ell\left(\omega\left(\frac{k}{\omega -1} -1\right)\right)+fr\tilde{r}^\ell\left(\frac{k}{\omega -1} +r - \omega\right)+gs\tilde{s}^\ell\left(\frac{k}{\omega -1} +s - \omega\right) \\ & \equiv 
        \left(\frac{k}{\omega-1}-\omega\right)(k\tilde{k}^\ell+fr\tilde{r}^\ell+gs\tilde{s}^\ell)+k^2\tilde{k}^\ell+fr^2\tilde{r}^\ell+gs^2\tilde{s}^\ell \\ & \equiv
        k^2\tilde{k}^\ell+fr^2\tilde{r}^\ell+gs^2\tilde{s}^\ell \\ & \equiv
        k^2\tilde{k}^\ell-r(k\tilde{k}^\ell+gs\tilde{s}^\ell)+gs^2\tilde{s}^\ell \\ & \equiv 
        k(k-r)\tilde{k}^\ell+gs(s-r)\tilde{s}^\ell
    \end{align*}
    So it suffices to show that $m$ divides $k(k-r)\tilde{k}^\ell+gs(s-r)\tilde{s}^\ell$ which we will again do by induction on $\ell \geq 0$. The base case is $k(k-r)+gs(s-r)=nk = m\omega(\omega-1)$. Assume $k(k-r)\tilde{k}^\ell+gs(s-r)\tilde{s}^\ell \equiv 0$ for some $\ell \geq 0$ and we get 
    \begin{align*}
        & k(k-r)\tilde{k}^{\ell +1}+gs(s-r)\tilde{s}^{\ell +1} \\ & \equiv
        k(k-r)\tilde{k}^\ell\left(\omega\left(\frac{k}{\omega -1} -1\right)\right)+gs(s-r)\tilde{s}^\ell\left(\frac{k}{\omega -1} +s - \omega\right) \\ & \equiv
        \left(\frac{k}{\omega -1}-\omega\right)(k(k-r)\tilde{k}^\ell+gs(s-r)\tilde{s}^\ell)+k^2(k-r)\tilde{k}^\ell+gs^2(s-r)\tilde{s}^\ell \\ & \equiv
        k^2(k-r)\tilde{k}^\ell+gs^2(s-r)\tilde{s}^\ell \\ & \equiv
        k^2(k-r)\tilde{k}^\ell-s(k(k-r)\tilde{k}^\ell) \\ & \equiv
        k(k-r)(k-s)\tilde{k}^\ell \\ &\equiv nk\mu\tilde{k}^\ell \hspace{8cm}\mbox{from equations (\ref{eq:srg}) and (\ref{eq:srgspec})}\\ & \equiv m\omega(\omega-1)\mu\tilde{k}^\ell \\ & \equiv 0
    \end{align*}
\end{proof}
While this shows we will not be able to eliminate any feasible parameter sets from possibly existing just by finding an integer $\ell$, we show in the next section that the fact about walk-regularity can be used to find other restrictive conditions on the existence of a certain family of strongly regular graphs.
\section{Locally Linear Strongly Regular Graphs}\label{sec:localsrgs}
An srg$(n,k,\lambda, \mu)$ is locally linear (3-clique regular) if and only if $\lambda=1$. There are many other families of strongly regular graphs that have the clique regular property which we will discuss in chapter 6, but the locally linear strongly regular graphs are interesting because they are the only ones for which the clique regular property is forced by their parameters. It follows from Theorem \ref{thm:rca1} that any srg$(n,k,1,\mu)$ is not only 3-clique regular but also a regular clique assembly.\par
As discussed in the introduction, the existence problem for these graphs was the main motivation for this investigation, and so in this section we explore multiple different approaches to try and find a necessary condition on the existence of these graphs. \par
Using Theorem \ref{thm:srg}, we can enumerate all non-boring locally linear strongly regular graphs that have strongly regular 3-clique graphs. It turns out there are only three.
 \begin{thm}\label{thm:3srgs}
     The only non-boring strongly regular locally linear graphs that have strongly regular $3$-clique graphs are the unique graphs with the parameters $\textnormal{srg}(9,4,1,2)$, $\textnormal{srg}(15,6,1,3)$, and $\textnormal{srg}(27, 10, 1, 5)$.
 \end{thm}
 \begin{proof}
     These graphs all have strongly regular $3$-clique graphs following from Theorem \ref{thm:srg}. So let $\Gamma$ be an srg$(n,k,1,\mu)$ such that $C_3(\Gamma)$ is strongly regular and we will show that $\Gamma$ is the unique strongly regular graph on one of the enumerated parameters. From Theorem \ref{thm:srg} either $k=6$ or $s=-\frac{k}{2}$ so first assume $k=6$. Then from $(n-k-1)\mu=k(k-\lambda-1)$ it follows that $n=\frac{24}{\mu}+7$. So since $\mu$ divides $24$, $\mu \in \{1, 2, 3, 4, 6, 8, 12, 24\}$. From (\ref{eq:srgspec}), the multiplicity of the largest eigenvalue of $\Gamma$ is \[f = \frac{1}{2}\left[\left(\frac{24}{\mu}+6\right)-\frac{12 + (\frac{24}{\mu}+6)(1-\mu)}{\sqrt{(1 - \mu)^2+4(6-\mu)}}\right],\] from which it follows that $f$ is only an integer when $\mu =3$. This shows that $\Gamma$ is the unique srg$(15,6,1,3)$.\par
     Now assume that $s=-\frac{k}{2}$. From (\ref{eq:srgspec}) we get
     \[s = \frac{1}{2}\left[(1 - \mu) - \sqrt{(1 - \mu)^2+4(k-\mu)}\right] = -\frac{k}{2},\] which implies that $\mu = \frac{k}{2}$. So from $(n-k-1)\mu=k(k-\lambda-1)$ we get that $\Gamma$ is an srg$(3(k-1), k, 1, \frac{k}{2})$ for some $k$. So the multiplicity of the smallest eigenvalue of $\Gamma$ is
     \[g = \frac{8(k-1)}{k+2}.\] Since $0<g<8$ when $k >1$,  $g \in \{1, \ldots, 7\}$ since the multiplicity must be an integer. Solving for $k$ gives $k\in \{2, 4, 6, 10, 22\}$ since $k$ must also be an integer. If $k=2$ this implies that $\Gamma \cong K_3$, a boring graph, and $k=22$ gives the parameters of $\Gamma$ as srg$(63, 22, 1, 11)$ which violate the absolute bound \cite{brouwer}. The remaining  $k\in \{4,6,10\}$ give the result.
 \end{proof}
Even if the 3-clique graph of an srg$(n,k,1,\mu)$ is not strongly regular, it still has some highly restrictive structural properties. For the remainder of this section, $\Gamma$ will represent a non-boring srg$(n,k,1,\mu)$ that is not one of the three enumerated in Theorem \ref{thm:3srgs}. From Theorem \ref{thm:srgspec}, it will have spectrum $k^1,r^f,s^g$ where these values are given by equations (\ref{eq:srgspec}). So $C_3(\Gamma)$ will be a connected $\rca(m,d,\omega)$ from Theorem \ref{thm:rca2} and will have spectrum $d^1,\r^f,\s^g,-3^{(m-n)}$ where 
\[m=\frac{nk}{6}, \quad \quad
    d = 3\left(\frac{k}{2}-1\right), \quad \quad
    \omega = \frac{k}{2},\]
\[\r=\frac{k}{2}+r-3, \quad \text{ and } \quad \s=\frac{k}{2}+s-3,\] from equation (\ref{eq:spec}).\par As we discussed in section \ref{sec:srgcrg}, $C_3(\Gamma)$ is walk-regular which means that the number of closed walks of length $\ell$ through some vertex in the graph is given by the expression 
\[\theta_\ell:=\frac{1}{m}\left(d^\ell+f\r^\ell+g\s^\ell+(m-n)(-3)^\ell \right)\]
and is independent of the vertex chosen. Using this idea, we can easily calculate the number of triangles and quadrangles through a given point respectively as
\[\Delta:=\frac{\theta_3}{2} \quad \text{ and } \quad\Xi :=\frac{\theta_4-2d^2+d}{2}.\]\par
So now fix some vertex $v$ in $C_3(\Gamma)$. The neighborhood of $v$ is fully understood since $C_3(\Gamma)$ is an $\rca(m,d,\omega)$, so we will study the vertices not adjacent to $v$. Define $T_i$ as the set of vertices not adjacent to $v$ that have exactly $i$ neighbors in common with $v$, and $\tau_i=|T_i|$ as the number of these vertices. Formally, for a vertex $u\in C_3(\Gamma)$, $u\in T_i$ if and only if $u\not \in N(v)$ and $|N(v)\cap N(u)|=i$. Following from Lemma \ref{lem:nonadj}, $\tau_i=0$ if $i>\frac{d}{\omega -1}=3$ so we need only to consider $T_0, T_1, T_2, T_3$. This setup can be visualized in Figure \ref{fig:hangingv}.\par
\begin{figure}[htbp]
    \centering
    \includegraphics[width=0.6\linewidth]{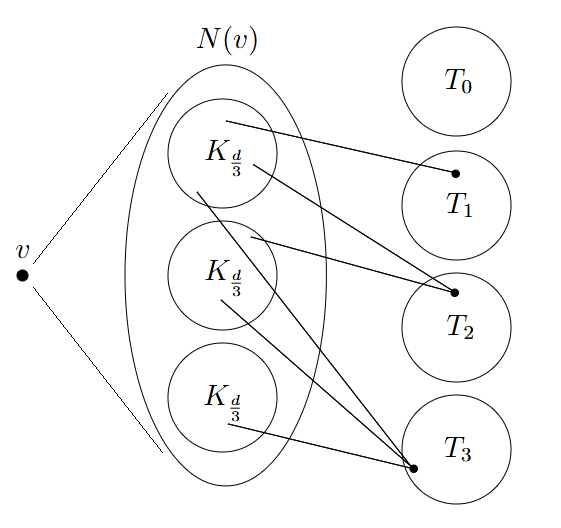}
    \caption[The neighborhood of vertex $v$ and the $T_i$ sets]{The neighborhood of $v$ and the $T_i$ sets}
    \label{fig:hangingv}
\end{figure}
The authors of \cite{dam} used a counting argument to derive a general system of linear equations that walk-regular graphs must satisfy, and for graph $C_3(\Gamma)$ this system is:
\begin{equation}\label{eq:system}\begin{split}
    \tau_0 + \tau_1 + \tau_2 + \tau_3 &= m-d-1, \\
    \tau_1 + 2\tau_2 +3\tau_3 &= d(d-1)-2\Delta, \\
    \tau_2 + 3\tau_3 &= \Xi - \binom{d/3-1}{2}d.
\end{split}\end{equation}
For a graph with the properties of $C_3(\Gamma)$ to exist, there must be at least one positive integer solution to system (\ref{eq:system}). We will expand this system by introducing 9 new positive integer variables and 9 new linear equations that they must satisfy. For $i,j=0,1,2,3$ we define $\rho_{i,j}=\rho_{j,i}$ as the number of edges with one endpoint in $T_i$ and the other in $T_j$. Note that $\rho_{i,i}$ does not double count edges, and is exactly the number of edges in the induced subgraph on vertices $T_i$. This implies there are 10 $\rho$ variables, but we can easily show that $\rho_{0,3}=0$.
\begin{prop}
    There are no edges between the sets $T_0$ and $T_3$.
\end{prop}
\begin{proof}
    Suppose to the contrary that $xy$ is an edge in $C_3(\Gamma)$ with $x\in T_0$ and $y\in T_3$, and consider $x$,$y$ and $v$ as triangles in $\Gamma$. Each vertex of triangle $y$ must share an edge with a vertex of $v$ since in $C_3(\Gamma)$ they have three common neighbors. And since $x$ is adjacent to $y$ in $C_3(\Gamma)$, triangles $x$ and $y$ must share some vertex in $\Gamma$. This means at least one vertex in triangle $x$ is adjacent to a vertex in triangle $v$, but this implies $x$ and $v$ share a common neighbor in $C_3(\Gamma)$, contradicting $x\in T_0$.
\end{proof}
We will refer to (\ref{eq:system}) along with the equations we derive below as the $\tau, \rho$ system. The first four equations in the $\tau, \rho$ system involve counting the total edges coming of off the vertices in the sets $T_0,T_1,T_2,T_3$. 
\begin{prop}\label{edgecount}
    We have
    \[(d-i)\tau_i=\rho_{i,i}+\sum_{j=0}^3\rho_{i,j}\]
    for each $i=0,1,2,3$.
\end{prop}
\begin{proof}
    A vertex not adjacent to $v$ that has $i$ neighbors in common with $v$ has $i$ of its edges going into $N(v)$ and so the remaining $d-i$ edges are incident with other vertices not adjacent to $v$. So these total edges are counted by the sum of $\rho_{i,j}$ but the edges between two vertices each sharing exactly $i$ common neighbors with $v$ is double counted.
\end{proof}
These four equations are clearly independent and we can verify with SAGE that so far, the whole $\tau,\rho$ system is also independent. The next four equations arise from counting the number of 3-walks between $v$ and vertices not adjacent to $v$.
\begin{prop}\label{prop:3walk}
    We have
    \[\left ( 9\mu +i\left(\frac{d}{3}-5 -\mu\right) \right)\tau_i=i\rho_{i,i}+\sum_{j=1}^3j\rho_{i,j}\]
    for each $i=0,1,2,3$.
\end{prop}
\begin{proof}
    This is equivalent to showing
    \[\left((9-i)\mu -4i +2i\left(\frac{k}{2}-2 \right) \right)\tau_i=i\left(\frac{d}{3}-1\right)\tau_i+i\rho_{i,i}+\sum_{j=1}^3j\rho_{i,j}.\]
    First, we will show that the coefficient on $\tau_i$ from the left hand side of this equation counts the number of 3-walks between $v$ and a vertex $u\in T_i$. Consider $u$ and $v$ as triangles in $\Gamma$. Of the 3 vertices in triangle $u$, $i$ of them are adjacent to a vertex of triangle $v$. Call the vertices of triangle $u=\{u_1,u_2,u_3\}$ and $v=\{v_1,v_2,v_3\}$, ordered so that $u_j$ is adjacent to $v_j$ for $j=1,\ldots, i$, where no vertices are adjacent if $i=0$.\par 
    Consider first a pair of adjacent vertices $v_j$ and $u_j$. The edge $u_jv_j$ is contained in a unique triangle denoted $C_{u_jv_j}$. The edges of $u_j$ are partitioned into $\frac{k}{2}$ triangles of which $u$ and $C_{u_jv_j}$ are two. In $C_3(\Gamma)$, walking from $u$ to one of these triangles, then to $C_{u_jv_j}$ and finally to $v$ is a 3-walk. There is the same number of 3-walks obtained by considering the edges at $v_j$ and there are $i$ pairs of adjacent vertices to consider, so we add the term $2i\left(\frac{k}{2}-2 \right)$ to the coefficient. \par
    Now consider a pair of non adjacent vertices $v_j$ and $u_h$ with $j\neq h$. These vertices share $\mu$ common neighbors in $\Gamma$ since it is an srg$(n,k,1,\mu)$, but depending on if $u_h$ or $v_j$ is adjacent to a vertex in $v$ or $u$ respectively, some of these common neighbors may be in $u$ or $v$. Let $x\in N(v_j)\cap N(u_h)$ a common neighbor that is not in triangle $u$ or $v$. The edges $xv_j$ and $xu_h$ are contained in unique triangles $C_{xv_j}$ and $C_{xu_h}$, and so there is the 3-walk in $C_3(\Gamma)$ $u\to C_{xu_h}\to C_{xv_j}\to v$. Breaking down the cases for each $i=0,1,2,3$ and each pair of non-adjacent vertices $v_j$ and $u_h$, we can count that the number of these 3-walks is given by the term $(9-i)\mu -4i$.\par
    The right hand side counts the total number of 3-walks between vertices in $T_i$ and $v$ as well but from the perspective of $C_3(\Gamma)$. Let $u\in T_i$, and for each of the $i$ different vertices $x\in N(u)\cap N(v)$, and the $\left( \frac{d}{3} -1\right)$ vertices $y \in N(x)\cap N(v)$ we have a 3-walk $u\to x\to y \to v$. So we add the term $i\left(\frac{d}{3}-1\right)\tau_i$.\par
    Finally, for each edge $ux$ with $u\in T_i$ and $x\in T_j$ for some $j=0,1,2,3$, there are $j$ 3-walks from $u$ to $x$ to each of the $j$ vertices in $N(x)\cap N(v)$ and then to $v$. However, we need to double count the edges within $T_i$ since a 3-walk could start at either endpoint. So we add a $j\rho_{i,j}$ term for each $j$ but an additional term for $i\rho_{i,i}$.
\end{proof}
While it is clear that these four equations are independent, using SAGE we can verify that the $\tau,\rho$ system with the 11 equations we have so far shown has only rank 10, meaning one of the four equations from Proposition \ref{prop:3walk} is a linear combination of the other 9 equations and does not provide any useful restriction.\par
The last equation comes from counting the number of closed 5-walks at vertex $v$. For this equation we must first count the number of 5-walks in $C_3(\Gamma)$ starting and ending at $v$ for which every vertex in the walk is either $v$ or adjacent to $v$.
\clearpage
\begin{lem}
    The number of closed 5-walks at $v$ for which every vertex in the walk is either $v$ or adjacent to $v$, is given by
    \[\bigstar := \frac{d^4}{27}+\frac{d^3}{3} - d^2-d.\]
\end{lem}
\begin{proof}
    The induced subgraph of $C_3(\Gamma)$ on vertices $v$ and $N(v)$ has adjacency matrix
    \[B:=\begin{bmatrix}
        0 & j^\top & j^\top & j^\top \\
        j & J_{\frac{d}{3}}-I_{\frac{d}{3}} & 0 & 0 \\
        j & 0 & J_{\frac{d}{3}}-I_{\frac{d}{3}} & 0 \\
        j & 0 & 0 & J_{\frac{d}{3}}-I_{\frac{d}{3}}
    \end{bmatrix}\]
    where $j$ is the all ones column vector of length $\frac{d}{3}$ and $jj^\top=J_{\frac{d}{3}}$ is the $\frac{d}{3} \times\frac{d}{3}$ all ones matrix. The number of closed 5-walks at vertex $v$ is given by $(B^5)_{1,1}$ which can be found when we notice that $j^\top j=\frac{d}{3}$. We omit a proof by hand because it is tedious and lengthy, but it can be verified with SAGE that $(B^5)_{1,1}=\bigstar$.
\end{proof}
\begin{prop} \label{prop:5walks}
    We have
    \[\sum_{i\leq j}ij\rho_{i,j}=\frac{\theta_5-\bigstar}{2}-\frac{4}{3}\Delta d-2\left(\frac{d}{3}-1 \right)\left(\Xi -\binom{d/3-1}{2}d \right).\]
\end{prop}
\begin{proof}
    This equation is equivalent to
    \[\theta_5=\bigstar+\frac{8}{3}\Delta d +4\left(\frac{d}{3}-1 \right)(\tau_2+3\tau_3)+2\sum_{i \leq j}ij\rho_{i,j}\]
    and since $\theta_5$ is the total number of closed 5-walks at vertex $v$, this can be shown with a counting argument. Since $\bigstar$ counts all 5-walks in which every vertex is $v$ or adjacent to $v$, we need now only count the 5-walks in which one or two vertices is not adjacent to $v$. For notational convenience, denote $N=\{v\}\cup N(v)$.\par
    Take some triangle at $v$ with other vertices $x$ and $y$. For each vertex $u\in N(x)-N$, there are the two closed 5-walks at $v$, $v\to x\to u\to x\to y\to v$ and the same walk in reverse order. By the structure of $C_3(\Gamma)$ apparent in Figure \ref{fig:hangingv}, we have $|N(x)-N|=d-|N(x)\cap N|=d-\frac{d}{3}=\frac{2}{3}d$. Since the same number of 5-walks comes from vertex $y$ and there are $\Delta$ triangles, we add the term $\frac{8}{3}\Delta d$.\par
    Now let $u\in T_i$ so that there are $i$ edges going from $u$ into $N(v)$. Choosing two of these $i$ edges, let $x,y \in N(v)$ be the vertices they connect to. For each of the $\frac{d}{3}-1$ vertices $z\in N(x)\cap N(v)$, there is the 5-walk $v\to z\to x\to u\to y\to v$ and its reversed order. Since the same number of 5-walks comes from vertex $y$ and there are $\binom{i}{2}$ ways to choose the edges off of $u$, we add the term $4\left(\frac{d}{3}-1\right)\binom{i}{2}\tau_i$ for each $i=0,1,2,3$. Since $\binom{0}{2}=\binom{1}{2}=0$, this becomes the term $4\left(\frac{d}{3}-1 \right)(\tau_2+3\tau_3)$, and we can replace $\tau_2+3\tau_3$ with $\Xi -\binom{d/3-1}{2}d$ from the original system (\ref{eq:system}). These first two terms capture every closed 5-walk with only one vertex not adjacent to $v$.\par
    And finally, let $xy$ be some edge of $C_3(\Gamma)$ with $x\in T_i$ and $y\in T_j$ so there are $i$ edges going from $x$ to $N(v)$ and $j$ edges going from $y$ to $N(v)$. Thus there are $ij$ different ways to choose some combination of these edges, and each combination corresponds to a closed 5-walk at $v$ and the reversed order walk. Since there are $\rho_{i,j}$ such edges, we add the term $2ij\rho_{i,j}$ for each $i,j=0,1,2,3$ and $i\leq j$. This term captures every closed 5-walk with two vertices not adjacent to $v$.
\end{proof}\par
Using SAGE, we can verify that this equation is also a linear combination of the other equations in the $\tau,\rho$ system, and so provides no useful information.\par
While these are theoretical proofs that $C_3(\Gamma)$ must satisfy the $\tau,\rho$ system, it is useful to verify that they hold for real graphs. There are four known constructions of srg$(n,k,1,\mu)$'s other than the three from Theorem \ref{thm:3srgs}. In table \ref{tab:taurho}, we show the parameters of these graphs, the parameters and spectrums of their 3-clique graphs, and the $\tau$ and $\rho$ values of these graphs numbered 1 to 4. We can use SAGE to show that these values do provide solutions to the $\tau,\rho$ system that we derived above.\par
The graphs 1,2 and 4 from table \ref{tab:taurho} are vertex transitive, which means that every vertex is symmetrical to every other vertex via a graph automorphism. As a consequence, fixing any vertex produces the same solution to the $\tau,\rho$ system. However graph 3, an srg$(378,52,1,8)$, is not vertex transitive and neither is its clique graph, and so fixing each vertex produces one of three different solutions to the $\tau,\rho$ system.\par

\begin{table}[htbp]
\centering
\begin{tabular}{|r|c|c|c|}
\hline
   \# & Parameters of $\Gamma$ & Parameters of $C_3(\Gamma)$ & Spectrum of $C_3(\Gamma)$ \\
\hhline{|====|}
1. & srg$(81,20,1,6)$ & rca$(270,27,10)$ & $27^1,9^{60},0^{20},-3^{189}$ \\
\hline
2. & srg$(243,22,1,2)$ & rca$(891,30,11)$ & $30^1,12^{132},3^{110},-3^{648}$ \\
\hline
3. & srg$(378,52,1,8)$ & rca$(3276,75,26)$ & $75^1,27^{273},12^{104},-3^{2,898}$ \\
\hline
4. & srg$(729,112,1,4)$ & rca$(13608,165,56)$ & $165^1,57^{616},30^{104},-3^{12,879}$ \\
\hline
\end{tabular}
\\[1em]
\begin{tabular}{|r||c|c|c|c|c|c|}
\hline
 \#   &  1.  &  2.  &  \multicolumn{3}{|c|}{3.}   &  4.  \\
 \hline
 $\tau_0$  &  $8$  &  $300$  &  $450$  &  $450$  &  $450$  &  $1,452$  \\
 \hline
 $\tau_1$  &  $0$  &  $540$  &  $1,800$  &  $1,800$  &  $1,800$  &  $5,940$  \\
 \hline
 $\tau_2$  &  $216$  &  $0$  &  $900$  &  $900$  &  $900$  &  $5,940$  \\
 \hline
 $\tau_3$  &  $18$  &  $20$  &  $50$  &  $50$  &  $50$  &  $110$  \\
 \hline
 $\rho_{0,0}$  &  $0$  &  $1800$  &  $3,525$  &  $3,569$  &  $3,375$  &  $21,780$  \\
 \hline
 $\rho_{0,1}$  &  $0$  &  $5400$  &  $21,000$  &  $20,824$  &  $21,600$  &  $130,680$  \\
 \hline
 $\rho_{0,2}$  &  $216$  &  $0$  &  $5,700$  &  $5,788$  &  $5,400$  &  $65,340$  \\
 \hline
 $\rho_{0,3}$  &  $0$  &  $0$  &  $0$  &  $0$  &  $0$  &  $0$  \\
 \hline
 $\rho_{1,1}$  &  $0$  &  $4,860$  &  $37,566$  &  $37,716$  &  $36,900$  &  $222,750$  \\
 \hline
 $\rho_{1,2}$  &  $0$  &  $0$  &  $35,136$  &  $35,064$  &  $36,000$  &  $392,040$  \\
 \hline
 $\rho_{1,3}$  &  $0$  &  $540$  &  $1,932$  &  $1,880$  &  $1,800$  &  $5,940$  \\
 \hline
 $\rho_{2,2}$  &  $2,376$  &  $0$  &  $11,664$  &  $11,604$  &  $11,250$  &  $249,480$  \\
 \hline
 $\rho_{2,3}$  &  $432$  &  $0$  &  $1,536$  &  $1,640$  &  $1,800$  &  $11,880$  \\
 \hline
 $\rho_{3,3}$  &  $0$  &  $0$  &  $66$  &  $40$  &  $0$  &  $0$ \\
 \hline
\end{tabular}
\caption{Known srg constructions and their solutions to the $\tau,\rho$ system}
\label{tab:taurho}
\end{table}
This non-vertex transitive graph is useful because it provides insight into possible expansions of the $\tau,\rho$ system. For one, since each solution has the same values for the $\tau$ variables, it seems likely that there exists a fourth linearly independent equation in the $\tau$ variables to extend the original system (\ref{eq:system}). Also, subtracting one solution vector from each of the other two, we can see that the null-space of any expansion of the $\tau,\rho$ system must have dimension at least two. This means there is at most one more linearly independent equation in the $\tau$ and $\rho$ variables that could be used to extend the system. This is a possible path for future research on this topic.\par
So we have shown for any srg$(n,k,1,\mu)$ to exist, there must be a solution in non-negative integers to the rank 10, 13 variable $\tau,\rho$ system. Although this appears to be a restrictive condition, it ultimately fails to limit the existence of these graphs. We use a Python algorithm discussed in Appendix \ref{ap:llsrg} to find all feasible parameters sets srg$(n,k,1,\mu)$ with degree $k$ not more than 50 million, and using a tool from Microsoft Research called Z3 \cite{z3}, we can show there exists solutions for all these parameter sets. So this condition does not eliminate any from possibility.

\chapter{Critical Groups}
For a graph $\Gamma=(V,E)$ with $n$ vertices and $c$ connected components, we define its \textbf{Laplacian} as $L=D-A$ where $D$ is $\Gamma$'s degree matrix and $A$ is its adjacency matrix. Viewing the Laplacian as a map $L:\ZZ^n\to \ZZ^n$, then $\text{coker}(L) = \ZZ^n/\text{Im}(L)\cong \ZZ^c \oplus \mathcal{K}(\Gamma)$ for some finite abelian group $\mathcal{K}(\Gamma)$ which we call the \textbf{critical group} of $\Gamma$.\par
It is known that the order of $\mathcal{K}(\Gamma)$ is equal to $\kappa(\Gamma)$, the number of spanning forests of the graph. This value $\kappa(\Gamma)$ is also equal to the determinant of any sub-matrix of $L$ obtained by removing the $i^{\textnormal{th}}$ row and column, and when the graph is connected this is equal to the product of the non-zero eigenvalues of the Laplacian divided by $n$, the number of vertices of $\Gamma$ \cite[Proposition 1.3.4]{haemers}. Since the Laplacian only has one zero eigenvalue when $\Gamma$ is connected, this is also equal to the expression $\frac{1}{n}\lim_{\lambda \to 0}\frac{(-1)^np(L,\lambda)}{-\lambda}$, which will be useful later.\par
There is a different formulation of the critical group of a graph that will turn out to be more useful for our purposes \cite[p. 319]{godsilbook}. Fix some arbitrary orientation of the edges of $\Gamma$, that is for each $\{x,y\} \in E$ define one vertex as directed towards the other and write them now as ordered pairs $(x,y)\in E$. This orientation can be arbitrary and up to isomorphism will not affect the critical group. \par
Now consider the free abelian group on the edges of $\Gamma$, $\ZZ^E$, with this fixed orientation. In order to examine edges oriented in the opposite direction, we define $(y,x)=-(x,y)$. For each vertex $v\in \Gamma$, the bond of that vertex is defined
\[b(v):=\sum_{u\in N(v)}(v,u)\in \ZZ^E\]
and the \textbf{bond space} $B\subseteq \ZZ^E$ of $\Gamma$ is defined as the integer span of the bonds of every vertex. We also define the \textbf{cycle space} $Z \subseteq \ZZ^E$ of $\Gamma$ as the integer span of all directed cycles of the graph, where a directed cycle is defined as
\[\sum_{j=1}^\ell(u_{j-1}, u_j) \in \ZZ^E\] where $u_{j-1}$ is adjacent to $u_j$, $u_0=u_\ell$, and all other vertices are distinct. It is a fact that $B$ and $Z$ are orthogonal with respect to the standard inner product.\par
With these definitions, it is possible to show that
\begin{equation}\label{eq:critical}\mathcal{K}(\Gamma)\cong \ZZ^E/(B\oplus Z)\end{equation}
which is much easier to deal with when defining maps between critical groups.
\section{The Critical Group of a Clique Subdivision Graph}
In this section, we will begin establishing connections between the critical groups of $\Gamma$ and $C_\omega(\Gamma)$ when $\Gamma$ is $\omega$-clique regular. To do this, it will be helpful to introduce a new graph transformation, the \textbf{$\omega$-clique subdivision graph}. This is defined so as to be the halfway transformation between $\Gamma$ and $C_\omega(\Gamma)$ and generalize the concept of the edge subdivision graph.\par
Denote the $\omega$-clique subdivision graph of $\Gamma$ as $S_\omega(\Gamma)$. The vertex set of $S_\omega(\Gamma)$ is the vertices of $\Gamma$ union the $\omega$-cliques of $\Gamma$, and the graph is bipartite on this partition with a vertex adjacent to a clique if and only if the clique contains that vertex. Figure \ref{fig:subdivision} shows an example of this construction.\par
\begin{figure}
    \centering
    \begin{subfigure}[c]{0.45\textwidth}
        \centering
        \includegraphics[width=0.8\textwidth]{gamma.png}
        \caption{A 3-clique regular graph $\Gamma$}
    \end{subfigure}
    \begin{subfigure}[c]{0.45\textwidth}
        \centering
        \includegraphics[width=0.8\textwidth]{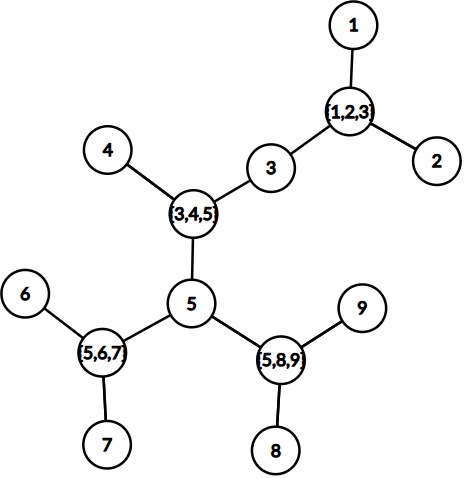}
        \caption{The graph $S_3(\Gamma)$}
    \end{subfigure}
    \caption{An example of the clique subdivision graph}
    \label{fig:subdivision}
\end{figure}
It should be clear by this definition, that if $R$ is the $\omega$-clique incidence matrix of $\Gamma$, then the adjacency matrix of $S_\omega(\Gamma)$ is $\begin{bmatrix}
    0 & R^\top\\
    R& 0
\end{bmatrix}$.
And if $\Gamma$ is $\omega$-clique regular with $n$ vertices and $m$ $\omega$-cliques, then the degree matrix of $S_\omega(\Gamma)$ is $\begin{bmatrix}
    \omega I_m & 0 \\
    0 & \frac{1}{\omega -1}D
\end{bmatrix}$ where $D$ is the degree matrix of $\Gamma$.\par
For the remainder of this chapter, $A$ will be the adjacency matrix of $\Gamma$, $D$ the degree matrix, $L=D-A$ the Laplacian, $R$ the $\omega$-clique incidence matrix, $V$ the vertex set, $E$ the edge set, and similarly for graphs $S_\omega(\Gamma)$ and $C_\omega(\Gamma)$ but with the subscripts $_S$ and $_C$ respectively.
\begin{thm}\label{thm:order}
    Suppose $\Gamma=(V,E)$ is an $\omega$-clique regular graph with $n$ vertices, $m$ $\omega$-cliques and $c$ connected components. Then 
    \begin{equation}\kappa(S_\omega(\Gamma))=\omega^{m-n+c}\kappa(\Gamma). \label{eq:spantree1}\end{equation}
    If $\Gamma$ is also $k$-regular, then
    \begin{equation}\kappa(C_\omega(\Gamma))=\left(\frac{k}{\omega-1}\right)^{m-n-c}\omega^{m-n+c}\kappa(\Gamma). \label{eq:spantree2}\end{equation}
\end{thm}
Note that this implies if $\Gamma$ is $k$-regular, then \[\kappa(C_\omega(\Gamma))=\left(\frac{k}{\omega-1}\right)^{m-n-c}\kappa(S_\omega(\Gamma))\] and it was this fact which motivated the definition of $S_\omega(\Gamma)$. For the remainder of this section, $n,m,$ and $c$ will be defined as they are in Theorem \ref{thm:order}.
\begin{proof}
    We will begin by showing equations (\ref{eq:spantree1}) and (\ref{eq:spantree2}) when $\Gamma$ is connected ($c=1$). The characteristic polynomial of the Laplacian of $S_\omega(\Gamma)$ is given by
    \begin{align*}
        (-1)^{m+n}p(L_S;\lambda)&= \begin{vmatrix}
                            (\omega-\lambda)I_m & -R^\top \\
                            -R & \frac{1}{\omega-1}D-\lambda I_n
                        \end{vmatrix} \\
        &= (\omega - \lambda)^m \begin{vmatrix} \frac{1}{\omega-1}D-\lambda I_n -\frac{1}{\omega - \lambda}RR^\top\end{vmatrix} & \mbox{from \cite[p. 114]{abadir2005matrix}}\\
        &= (\omega - \lambda)^{m-n} \begin{vmatrix} \frac{\omega-\lambda}{\omega-1}D-\lambda(\omega -\lambda) I_n -RR^\top\end{vmatrix} \\
        &= (\omega - \lambda)^{m-n} \begin{vmatrix} \frac{\omega-\lambda}{\omega-1}D-\lambda(\omega -\lambda) I_n -A-\frac{1}{\omega -1} D\end{vmatrix} & \mbox{by Lemma \ref{lem:blockmat}}\\
        &= (\omega - \lambda)^{m-n} \begin{vmatrix}D-A-\lambda(\omega -\lambda) I_n -\frac{\lambda}{\omega -1} D\end{vmatrix}.
    \end{align*}
    For brevity define $M(\lambda):=D-A-\lambda(\omega -\lambda) I_n -\frac{\lambda}{\omega -1} D$. So we get
    \begin{align*}
        \kappa(S_\omega(\Gamma))&= \frac{1}{m+n}\lim_{\lambda \to 0}\frac{(-1)^{m+n}p(L_S;\lambda)}{-\lambda} \\
        &= \frac{\omega^{m-n}}{m+n}\lim_{\lambda \to 0}\frac{|M(\lambda) |}{-\lambda} \\
        &= \frac{-\omega^{m-n}}{m+n}\lim_{\lambda \to 0}\text{tr}\left(\text{adj}(M(\lambda)) \left(\frac{-1}{\omega -1}D+(2\lambda -\omega)I_n\right)\right) 
        \end{align*}
        where this last equality follows from L'hopitals rule, and since the derivative of the determinant $|M(\lambda)|$ is given by Jacobi's formula \cite[p. 149]{magnus}.\\Carrying on and letting $v_1,\ldots,v_n$ be the vertices of $\Gamma$ and we have
        \begin{align*}
        \kappa(S_\omega(\Gamma))&= \frac{-\omega^{m-n}}{m+n}\text{tr}\left(\text{adj}(D-A) \left(\frac{-1}{\omega -1}D -\omega I_n\right)\right) \\
        &= \frac{-\omega^{m-n}}{m+n}\text{tr}\begin{bmatrix} \kappa(\Gamma)\left(\frac{-d(v_1)}{\omega -1}-\omega\right)  & &  \\
        & \ddots & \\
        & & \kappa(\Gamma)\left(\frac{-d(v_n)}{\omega -1}-\omega\right)
        \end{bmatrix} \\
        &= \frac{-\omega^{m-n}}{m+n}\sum_{v_i\in V} \kappa(\Gamma)\left(\frac{-d(v_i)}{\omega -1}-\omega\right) \\
        &= \frac{\omega^{m-n+1}\kappa(\Gamma)}{m+n}\left(\frac{\Sigma_{v_i\in V} d(v_i)}{\omega(\omega -1)}+n\right) \\
        &= \frac{\omega^{m-n+1}\kappa(\Gamma)}{m+n}\left(m+n\right) = \omega^{m-n+1}\kappa(\Gamma)
    \end{align*}
    since the sum of all degrees in a graph is twice the number of edges. Similarly, the characteristic polynomial of the Laplacian of $C_\omega(\Gamma)$ is given by
    \begin{align*}
        (-1)^mp(L_C;\lambda) &= p\left(A_C;\;\omega \left(\frac{k}{w-1}-1 \right) -\lambda\right) & \mbox{since $C_\omega(\Gamma)$ is regular}\\
        &= \left( \omega \left(\frac{k}{w-1} \right) -\lambda \right)^{m-n}p\left(A;\; k -\lambda\right) & \mbox{by Theorem \ref{thm:eigen}}\\
        &= \left( \omega \left(\frac{k}{w-1} \right) -\lambda \right)^{m-n}(-1)^n p\left(L;\;\lambda\right). & \mbox{since $\Gamma$ is regular}
    \end{align*}
    So we have 
    \begin{align*}
        \kappa(C_\omega(\Gamma)) &= \frac{1}{m} \lim_{\lambda \to 0}\frac{(-1)^{m}p(L_C;\lambda)}{-\lambda} \\
        &= \frac{\omega(\omega -1)}{nk} \left(\frac{k}{\omega -1} \right) ^{m-n}\omega^{m-n} \lim_{\lambda \to 0}\frac{(-1)^{n}p(L;\lambda)}{-\lambda} \\
        &= \left(\frac{k}{\omega -1} \right) ^{m-n-1}\omega^{m-n+1} \frac{1}{n}\lim_{\lambda \to 0}\frac{(-1)^{n}p(L;\lambda)}{-\lambda} \\
        &= \left(\frac{k}{\omega -1} \right) ^{m-n-1}\omega^{m-n+1} \kappa(\Gamma).
    \end{align*}
    For any $c$, the result now follows from the fact that the number of spanning forests equals the product of the number of spanning trees of each connected component of a graph. 
\end{proof}
This relationship between the orders suggests a deeper relationship between the groups and is what motivated this investigation. It was known for the case of the line graph and edge subdivision graph ($\omega =2$) but we have generalized this to the $\omega$-clique graph and $\omega$-clique subdivision graph for any $\omega\geq 2$.\par
The authors of \cite{reiner} used a known relationship between the critical group of a graph and its edge subdivision graph (2-clique subdivision graph), to derive a relationship between a graph and its line graph (2-clique graph). We will generalize their argument as far as it is possible to any $\omega\geq2$, to attempt to generalize the relationship between the critical group of an $\omega$-clique regular graph and its subdivision graph and clique graph.\par
We begin by examining the relationship between $\mathcal{K}(\Gamma)$ and $\mathcal{K}(S_\omega(\Gamma))$. Let $B$ and $Z$ represent the bond and cycle space of $\Gamma$ respectively, with $B_S$ and $Z_S$ representing them for $S_\omega(\Gamma)$. In general, anything with no subscript relates to $\Gamma$, and anything with the subscript $_S$ relates to $S_\omega(\Gamma)$.\par
For basis vectors $(v,C)\in E_S$, with $v$ a vertex in $\Gamma$ and $C$ an $\omega$-clique containing $v$, and $(x,y)\in E$ for arbitrary adjacent vertices $x$ and $y$, define the $\ZZ$-module homomorphisms
\begin{align*}
    h:\mathbb{Z}^{E_S} &\rightarrow \mathbb{Z}^E \\
    (v,C) &\mapsto \sum_{u \in C-\{v\}}(v,u) \\
    \\
    h^\top:\mathbb{Z}^E &\rightarrow \mathbb{Z}^{E_S} \\
    (x,y) &\mapsto (x,C_{xy}) +(C_{xy},y)
\end{align*} where $C_{xy}$ is the unique $\omega$-clique containing edge $(x, y)$.
\clearpage
\begin{lem}
    The functions $h$ and $h^\top$ are adjoint.
\end{lem}
\begin{proof}
    It suffices to show for arbitrary basis elements $(x,y)$ and $(u,C)$ of $\mathbb{Z}^{E}$ and $\mathbb{Z}^{E_S}$ respectively, that
    \begin{align*}
    \langle h(u,C), (x,y) \rangle &= \sum_{v \in C - \{ u \}} \langle (u,v), (x,y) \rangle \\
    &= \begin{cases}
        1 & \text{if $C = C_{xy}$ and $u = x$} \\
        -1 & \text{if $C = C_{xy}$ and $u = y$} \\
        0 & \text{otherwise}
    \end{cases} \\
    &= \langle (u, C), (x, C_{xy}) \rangle + \langle (u, C), (C_{xy}, y) \rangle \\
    &= \langle (u,C), h^\top(x,y)\rangle
\end{align*}
where $\langle \cdot,\cdot\rangle$ is the standard inner product.
\end{proof}
Using the critical group formulation in equation (\ref{eq:critical}), we can show that $h$ and $h^\top$ induce homomorphisms between the critical groups $\mathcal{K}(\Gamma)$ and $\mathcal{K}(S_\omega(\Gamma))$. That is, the maps defined by $x+(B\oplus Z) \mapsto h^\top(x)+(B_S\oplus Z_S)$ and $y+(B_S\oplus Z_S) \mapsto h(y)+(B\oplus Z)$ are well defined homomorphisms.
\begin{lem}
    The functions $h$ and $h^\top$ induce homomorphisms,
    \[h:\mathcal{K}(S_\omega(\Gamma))\rightarrow \mathcal{K}(\Gamma)\]
    \[h^\top:\mathcal{K}(\Gamma)\rightarrow \mathcal{K}(S_\omega(\Gamma)).\]
\end{lem}
\begin{proof}
    It is clear from equation (\ref{eq:critical}) that we must show $h(B_S \oplus Z_S) \subseteq B\oplus Z$ and $h^\top(B \oplus Z) \subseteq B_S\oplus Z_S$. Since $h$ and $h^\top$ are adjoint and $B$ is orthogonal to $Z$, it suffices to show that $h^\top(B)\subseteq B_S$ and $h^\top(Z) \subseteq Z_S$. Since $B$ is spanned by the bonds of $\Gamma$, let $v\in V$ and we will show $h^\top(b(v)) \in B_S$. We have
    \begin{align*}
        h^\top(b(v)) &= h^\top\left(\sum_{v \in N(v)}(v,u) \right) \\
        &= \sum_{u \in N(v)}(v,C_{vu})+\sum_{u \in N(v)}(C_{vu},u) \\
        &= (\omega-1)\sum_{C \in N_S(v)}(v,C)+\sum_{C \in N_S(v)}\left(\sum_{u \in N_S(C)}(C,u) \right)-\sum_{C \in N_S(v)}(C,v) \\
        &= \omega b_S(v)+\sum_{C \in N_S(v)}b_S(C) \\
        &\in B_S.
    \end{align*}

    Similarly, let $z=\sum_{j=1}^\ell(u_{j-1}, u_j)$ be a directed cycle in $\Gamma$ and we will show $h^\top(z)\in Z_S$. We have 
    \begin{align*}
        h^\top(z) &= h^\top \left(\sum_{j=1}^\ell(u_{j-1}, u_j) \right) \\
        &= \sum_{j=1}^\ell\Big((u_{j-1}, C_{u_{j-1}u_{j}})+(C_{u_{j-1}u_{j}}, u_j) \Big) \\
        &\in Z_S.
    \end{align*}
    So the linear functions $h$ and $h^\top$ induce homomorphisms between the respective critical groups, which we also denote with $h$ and $h^\top$.
\end{proof}
Notice from Theorem \ref{thm:order}, for groups $\mathcal{K}(\Gamma)$ and $\mathcal{K}(S_\omega(\Gamma))$, which group is bigger entirely depends on the sign of $m-n+c$. What testing with SAGE indicates is that if $m-n+c \geq 0$, then $h$ is onto and if $m-n+c\leq 0$, then $h^\top$ is onto. In fact, in the relevant cases it seems that $\ker(h) \cong (\ZZ/\omega\ZZ)^{m-n+c}$ or $\ker(h^\top)\cong (\ZZ/\omega\ZZ)^{n-m-c}$. If this were the case, it would show the following short exact sequences
\[0\rightarrow (\ZZ/\omega\ZZ)^{m-n+c} \rightarrow \mathcal{K}(S_\omega(\Gamma)) \xrightarrow{h} \mathcal{K}(\Gamma) \rightarrow 0,\]
\[0\rightarrow (\ZZ/\omega\ZZ)^{n-m-c} \rightarrow \mathcal{K}(\Gamma) \xrightarrow{h^\top} \mathcal{K}(S_\omega(\Gamma)) \rightarrow 0\] 
if $m-n+c\geq0$ or $\leq0$ respectively. Implying $\mathcal{K}(\Gamma)\cong \mathcal{K}(S_\omega(\Gamma))/\ker(h)$ or $\mathcal{K}(S_\omega(\Gamma))\cong \mathcal{K}(\Gamma)/\ker(h^\top)$ in the relevant cases.\par
This is where the generalization differs from the argument presented in \cite{reiner}. In the case of the line graph $(\omega=2)$ it can be shown that $m-n+c\geq0$ always. But in the generalization, this is not the case, and so we must do twice as much work.\par
Starting with the case $m-n+c \geq 0$, to show $\ker(h) \cong (\ZZ/\omega\ZZ)^{m-n+c}$ we claim it suffices to show: 
\begin{enumerate}
    \item[(1)] the order of every element in $\ker(h)$ divides $\omega$,
    \item[(2)] $\omega^{m-n+c}$ divides $ |\ker(h)|$ , and
    \item[(3)] $\ker(h)$ can be generated by $m-n+c$ elements.
\end{enumerate}\par
By the Fundamental Theorem of Finite Abelian Groups, $\ker(h)\cong \bigoplus_{i=1}^k(\ZZ/d_i\ZZ)$ for some natural numbers $k$ and $d_i>1$ such that $d_i$ divides $d_{i+1}$ for all $i$. If (1) can be shown, it would require that $d_i$ divides $\omega$ for all $i$; $d_i\leq \omega$. So if (3) can be shown, it would imply that $k\leq m-n+c$, and thus \begin{equation}\label{eq:kernel} |\ker(h)|=\prod_{i=1}^k d_i \leq \omega^k \leq \omega^{m-n+c}.\end{equation} Given this, showing (2) would imply that $|\ker(h)|=\omega^{m-n+c}$. So if $d_i<\omega$ for any $i$, the first inequality in (\ref{eq:kernel}) would be strict, and if $k<m-n+c$, the second inequality in (\ref{eq:kernel}) would be strict. Since either case would imply $|\ker(h)|<\omega^{m-n+c}$, we must have $d_i=\omega$ for all $i $ and $k=m-n+c$. Thus $\ker(h)\cong(\ZZ/\omega\ZZ)^{m-n+c}$.\par
A similar argument can be made to show $\ker(h^\top) \cong (\ZZ/\omega\ZZ)^{n-m-c}$ in the case $m-n+c \leq 0$.\par
We will begin by showing (1) in both cases.
\begin{lem}\label{lem:scalar}
    The composite functions 
    \[h h^\top:\mathcal{K}(\Gamma) \rightarrow \mathcal{K}(\Gamma)\]
    \[h^\top  h:\mathcal{K}(S_\omega(\Gamma)) \rightarrow \mathcal{K}(S_\omega(\Gamma))\] are both scalar multiplication by $\omega$.
\end{lem}
\begin{proof}
    It suffices to show this for arbitrary basis elements $(x,y)$ and $(v,C)$ of $\mathbb{Z}^{E}$ and $\mathbb{Z}^{E_S}$ respectively. First we have
    \begin{align*}
        h^\top h(v,C) &=h^\top \left( \sum_{u\in C-\{v\}}(v,u)\right) \\
        &=\sum_{u\in C-\{v\}} \Big((v, C)+(C,u)\Big) \\
        &= (\omega-1)(v,C)+\sum_{u\in C-\{v\}}(C,u)  \\
        &= \omega(v,C)+\sum_{u\in C}(C,u)  \\
        &= \omega(v,C)+b_S(C) \\
        &= \omega(v,C) \pmod{B_S}.
        \end{align*}
    And we also have
    \begin{align*}
        hh^\top(x,y) &=h\Big( (x,C_{xy})+(C_{xy},y) \Big) \\
        &= \sum_{u\in C_{xy}-\{x\}}(x,u)+\sum_{w\in C_{xy}-\{y\}}(w,y) \\
        &= 2(x,y)+\sum_{u\in C_{xy}-\{x,y\}} \Big( (x,u) +(u,y) \Big) \\
        &= \omega(x,y) +(\omega -2)(y,x)+\sum_{u\in C_{xy}-\{x,y\}} \Big( (x,u) +(u,y) \Big) \\
        &= \omega(x,y) +\sum_{u\in C_{xy}-\{x,y\}} \Big( (x,u) +(u,y) +(y,x) \Big) \\
        &= \omega(x,y) \pmod{Z}.
    \end{align*}
\end{proof}
\begin{cor}
    The order of each element in $\ker(h)$ and $\ker(h^\top)$, divides $\omega$.
\end{cor}
\begin{proof}
    Let $x \in \ker(h)$ and $y \in \ker(h^\top)$. Then using Lemma \ref{lem:scalar}
    \[\omega x = h^\top h(x) = h^\top(0) =0\quad \text{and} \quad \omega y = hh^\top(y) = h(0) =0.\]
    By Lagrange's Theorem, this implies the order of $x$ and the order of $y$ both divide $\omega$.
\end{proof}\par
Now we will show (2) in both cases.
\begin{lem}
    If $m-n+c\geq 0$, then $\omega^{m-n+c}$ divides $|\ker(h)|$ and, if $m-n+c\leq 0$, then $\omega^{n-m-c}$ divides $|\ker(h^\top )|$.
\end{lem}
\begin{proof}
    We will start with the first case so assume $m-n+c \geq 0$ which implies that $\omega^{m-n+c}$ is an integer. Since $\im(h)$ is a subgroup of $\mathcal{K}(\Gamma)$, by Lagrange's Theorem $|\im(h)|$ divides $|\mathcal{K}(\Gamma)|$ so there exists an integer $a$ such that $a|\im(h)|=|\mathcal{K}(\Gamma)|$. So since the First Isomorphism Theorem implies $\mathcal{K}(S_\omega(\Gamma))/\ker(h) \cong \im(h)$ and Theorem \ref{thm:order} states $|\mathcal{K}(S_\omega(\Gamma))|=\omega^{m-n+c}|\mathcal{K}(\Gamma)|$, we have
    \begin{align*}
        \frac{|\mathcal{K}(S_\omega(\Gamma))|}{|\ker(h)|}&=|\im(h)| \\
        \frac{a|\mathcal{K}(S_\omega(\Gamma))|}{|\ker(h)|}&=|\mathcal{K}(\Gamma)| \\
        a\omega^{m-n+c}|\mathcal{K}(\Gamma)|&=|\ker(h)||\mathcal{K}(\Gamma)| \\
        a\omega^{m-n+c}&=|\ker(h)|.
    \end{align*}
    Thus $\omega^{m-n+c}$ divides $|\ker(h)|$. In the case of $m-n+c \leq 0$, $\omega^{n-m-c}$ is an integer and Theorem \ref{thm:order} implies $|\mathcal{K}(\Gamma)|=\omega^{n-m-c}|\mathcal{K}(S_\omega(\Gamma))|$. So a similar argument applies to show $\omega^{n-m-c}$ divides $|\ker(h^\top )|$.
\end{proof}
It now remains to show (3), that $\ker(h)$ can be generated by $m-n+c\geq0$ elements and $\ker(h^\top)$ can be generated by $n-m-c\geq 0$ elements. This also differs from the argument of the authors of \cite{reiner}. In the case of the edge subdivision graph $(\omega =2)$, it is provable by an induction argument that both $\mathcal{K}(\Gamma)$ and $\mathcal{K}(S_2(\Gamma))$ can be generated with $m-n+c$ elements since $m$ is the number of edges in this case. However, this fact does not generalize to any $\omega \geq 2$, and using SAGE we can easily find examples of graphs that require more generators. It is the same in the case that $m-n+c \leq 0$.\par
We leave this as a path for future research, along with examining the relationship between $\mathcal{K}(S_\omega(\Gamma))$ and $\mathcal{K}(C_\omega(\Gamma))$.

\chapter{Examples}
We will conclude with a discussion of certain families of graphs that have the clique regular property. These examples are widely studied in graph theory and represent concrete examples of the graph properties that we have been studying only theoretically until this point. Examining these families allows for a better understanding of how our results manifest in specific cases. 
\section{Orthogonal Array Block Graphs}
 An \textbf{orthogonal array}, denoted $OA(n,m)$, is an $n^2 \times m$ array with entries from an $n$-element set with the property that the rows of any $n^2 \times 2$ sub array consist of all $n^2$ possible pairs exactly once. The \textbf{block graph} of an orthogonal array is the graph with vertices as the $1 \times m$ row vectors of the $OA(n,m)$, where two vectors are adjacent if and only if they have the same entry in some coordinate position. This means they share the same entry in exactly one position, since by the construction of the orthogonal array no two vectors can share the same entry in more than one position. Figure \ref{fig:oa_block} shows an example of an orthogonal array and its block graph.\par
    \begin{figure}
    \begin{subfigure}{.5\textwidth}
      \centering

\begin{tabular}{|c|c|c|}
\hline
$1$ & $1$ & $1$ \\
\hline
$1$ & $2$ & $2$ \\
\hline
$1$ & $3$ & $3$ \\
\hline
$2$ & $1$ & $2$ \\
\hline
$2$ & $2$ & $3$ \\
\hline
$2$ & $3$ & $1$ \\
\hline
$3$ & $1$ & $3$ \\
\hline
$3$ & $2$ & $1$ \\
\hline
$3$ & $3$ & $2$ \\
\hline
\end{tabular}

      \caption{An OA$(3,3)$ and}
    
    \end{subfigure}
    \begin{subfigure}{.5\textwidth}
      \centering
      \includegraphics[width=0.7\linewidth]{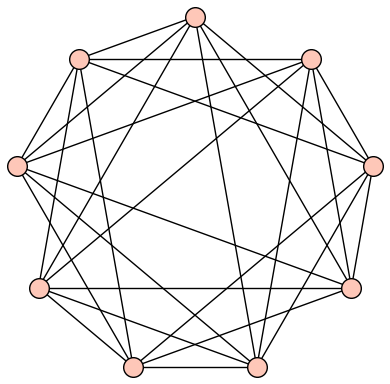}
      \caption{its Block Graph}
      
    \end{subfigure}
    \caption[An example of an orthogonal array and its block graph]{}
    \label{fig:oa_block}
    \end{figure}
A clique in an orthogonal array block graph is called a \textbf{canonical clique}, denoted $S_{ri}$, if every vector in the clique shares the same entry $i$ in the same column $r$. Clearly every maximal clique of this form has order $n$.
\begin{lem} \label{lem: OA}
    If $\Gamma$ is the block graph of  an $OA(n,m)$ and $n>(m-1)^2$, then every clique of order $n$ is a canonical clique.
\end{lem}
The proof of this lemma comes from \cite[pg. 99]{godsil} Corollary 5.5.3.
\begin{cor} \label{thm:OA}
    If $\Gamma$ is the block graph of  an $OA(n,m)$ and $n > (m-1)^2$, then $\Gamma$ is $n$-clique regular and $C_n(\Gamma)$ is isomorphic to the complete $m$-partite graph where each independent set of vertices has order $n$.
\end{cor}
\begin{proof}
    If two vectors $v$ and $u$ are adjacent in $\Gamma$, then they share the same entry $i$ in some column $r$. This is the only entry that the two vectors share by the construction of the orthogonal array. So the only canonical clique the edge $vu$ is in is $S_{ri}$ and by Lemma \ref{lem: OA} this is the only $n$-clique the edge is in.\par
    Also by Lemma \ref{lem: OA}, the canonical cliques of $\Gamma$ form all the vertices in $C_n(\Gamma)$. Notice that if $i\neq j$ then for all columns $r$, $S_{ri}$ and $S_{rj}$ have an empty intersection so are not adjacent in $C_n(\Gamma)$. Also notice that by the definition of an orthogonal array, if $r\neq t$ then for all $i$ and $j$, $S_{ri}$ and $S_{tj}$ share the vector with entry $i$ in column $r$ and entry $j$ in column $t$ and so are adjacent in $C_n(\Gamma)$. So there are $m$ independent sets of order $n$ in $C_n(\Gamma)$ and each vertex in each set is adjacent to every other vertex of every other independent set. Thus, $C_n(\Gamma)$ is isomorphic to the complete $m$-partite graph where each independent set of vertices has order $n$.
\end{proof}
A \textbf{square rook graph} is the graph with vertices as the points in a square grid where two points are adjacent if and only if they are in the same row or column.
\begin{cor} \label{thm:rook}
    If $\Gamma$ is the square rook graph on $n^2$ vertices, then $\Gamma$ is $n$-clique regular and $C_n(\Gamma) \cong K_{n,n}$.
\end{cor}
\begin{proof} 
    Since the square rook graph on $n^2$ vertices is the block graph of an $OA(n,2)$, the result follows from Corollary \ref{thm:OA}.\\
    The result also follows from Theorems \ref{thm:4cr} - \ref{thm:3ci} since the square rook graph on $n^2$ vertices is the line graph of $K_{n,n}$ which is $n$-regular and triangle-free.\\
    It also follows using the graph's spectrum since the square rook graph on $n^2$ vertices is an srg$(n^2, 2(n-1), n-2, 2)$ with spectrum \[2(n-1)^1,(n-2)^{2(n-1)}, -2^{(n-1)^2}.\] Using Remark \ref{rem:eigen}, the spectrum of $C_n(\Gamma)$ is \[n^1, 0^{2(n-1)}, -n^1\] and a classical result in spectral graph theory states this spectrum uniquely determines that $C_n(\Gamma) \cong K_{n,n}$.
\end{proof}
\section{Triangular Graphs}
 A \textbf{triangular graph}, denoted $T_n$, is the line graph of the complete graph on $n$ vertices, $K_n$. 
\begin{cor}
    If $n = 3$ or $n>4$, then $T_n$ is $(n-1)$-clique regular and $C_{n-1}(T_n) \cong K_n$.\\
\end{cor}
\begin{proof}
    This follows from Theorems \ref{thm:4cr} and \ref{thm:4ci} since $K_n$ is $(n-1)$-regular. Note that $T_4$ is not included since $K_4$ is $3$-regular but not triangle-free.\\
    Alternatively we can prove this using the spectrum of $T_n$ which is \[2(n-2)^1,\;(n-4)^{n-1},\;-2^{\frac{n(n-3)}{2}}\]
    because $T_n$ is an srg$(\frac{1}{2}n(n - 1),\; 2(n - 2),\; n - 2,\; 4)$.
    Also since $T_n$ is also $(n-1)$-clique regular and $2(n-2)$-regular, Remark \ref{rem:eigen} gives the spectrum of $C_{n-1}(T_n)$ as \[ (n-1)^1,\; -1^{n-1}.\]
    A well known result in spectral graph theory is that the only regular graph with this spectrum is the complete graph on $n$ vertices. So $C_{n-1}(T_n) \cong K_n$.
\end{proof}

\section{Generalized Quadrangle Collinearity Graphs}
 A \textbf{Generalized Quadrangle} $GQ(s,t)$ is a point-line incidence structure satisfying the following properties for some $s,t \geq 1$: \begin{itemize}
    \item every line has $s+ 1$ points,
    \item every point lies on $t + 1$ lines,
    \item there is at most one point common to any two distinct lines, and
    \item if point $P$ is not on line $\ell$, then there is a unique line incident with $P$ meeting $\ell$.
\end{itemize}
 The \textbf{collinearity graph} of a $GQ(s,t)$ is the graph where the vertices are the set of points in the $GQ(s,t)$ where two points are adjacent if and only if there is some line in the $GQ(s,t)$ that they are commonly incident to. Figure \ref{fig:gq} shows an example of a generalized quadrangle and it's collinearity graph.
\begin{figure}
    \begin{subfigure}{.5\textwidth}
      \centering
    \includegraphics[width=0.7\linewidth]{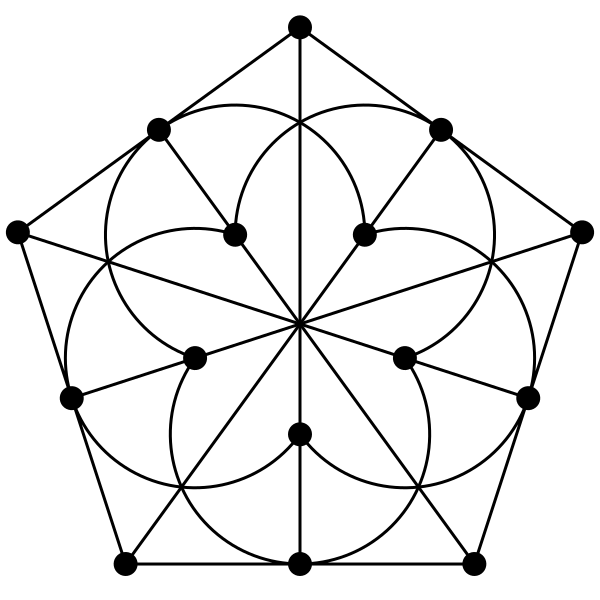}
      \caption{A GQ$(2,2)$ and}
      
    \end{subfigure}
    \begin{subfigure}{.5\textwidth}
      \centering
      \includegraphics[width=0.7\linewidth]{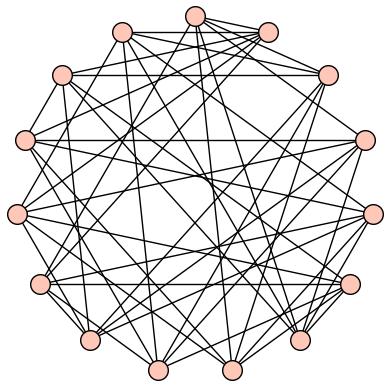}
      \caption{its Collinearity Graph}
      
    \end{subfigure}
    \caption[An example of a generalized quadrangle and its collinearity graph]{}
    \label{fig:gq}
    \end{figure}
    
\begin{cor} \label{thm:gq}
    Suppose $\Gamma$ is the collinearity graph of a $GQ(s,t)$. Then $\Gamma$ is $(s+1)$-clique regular and a regular clique assembly.
\end{cor}
\begin{proof}
     Observe that the $s+1$ points on any line in the $GQ(s,t)$ induce an $(s+1)$-clique in $\Gamma$. It follows from the definition that generalized quadrangles contain no triangles, so the only $(s+1)$-cliques in $\Gamma$ are those induced by the set of points on some line. So if points $v$ and $u$ are adjacent in $\Gamma$ then they are commonly incident to some line $\ell$ in $GQ(s,t)$. So the edge $uv$ is in the $(s+1)$-clique induced by the points on line $\ell$ and points $u$ and $v$ are commonly incident to no other lines, so the edge $uv$ is in a unique $(s+1)$-clique in $\Gamma$.\\
    Since the collinearity graph of a $GQ(s,t)$ is an srg$((s+1)(st+1), s(t+1), s-1, t+1)$, then $\Gamma$ is a regular clique assembly by Theorem \ref{thm:rca1}.
\end{proof}
For our final result, recall that the \textbf{dual} of a point line incidence structure is the structure formed by swapping the points and the lines of the original structure. If $\beta$ is a point line incidence structure, then point $P$ is on line $\ell$ in $\beta$ if and only if point $\ell$ is on line $P$ in the dual of $\beta$. In the case of generalized quadrangles, it turns out that the dual of a $GQ(s,t)$ is a $GQ(t,s)$ and so has a strongly regular collinearity graph. The following theorem gives an alternate proof that the dual of a $GQ(s,t)$ has a strongly regular collinearity graph.
\clearpage
\begin{cor} \label{thm:gq1}
    If $\Gamma$ is the collinearity graph of a $GQ(s,t)$, then $C_{s+1}(\Gamma)$ is isomorphic to the collinearity graph of the dual of the $GQ(s,t)$ structure that formed $\Gamma$ and is strongly regular with parameters 
    \[\textnormal{srg}(\left(t+1\right)\left(st+1\right), \; t(s+1),\; t-1, \;s+1).\]
\end{cor}
\begin{proof}
    Let $\Gamma$ be the collinearity graph of a $GQ(s,t)$. As shown in Corollary \ref{thm:gq}, there exists a bijection between the lines of the $GQ(s,t)$ and the $(s+1)$-cliques of $\Gamma$ that sends a line in the $GQ(s,t)$ to the set of points in $\Gamma$ that are incident to that line. Since the lines in the $GQ(s,t)$ form the points in the structure's dual, this is a bijection between the vertices of $C_{s+1}(\Gamma)$ and the points on the collinearity graph of the dual of the $GQ(s,t)$. We will show this bijection is an isomorphism. If two points $u$ and $v$ are adjacent in the collinearity graph of the dual of the $GQ(s,t)$, then they are commonly incident to some line $\ell$ in the dual of the $GQ(s,t)$. This means that point $\ell$ is incident to lines $u$ and $v$ in the $GQ(s,t)$ and so the $(s+1)$-cliques they form in $\Gamma$ share a vertex and are adjacent in $C_{s+1}(\Gamma)$. If these two points are not adjacent in the collinearity graph of the dual of the $GQ(s,t)$, then there is no line that they are commonly incident to. So in the $GQ(s,t)$, lines $u$ and $v$ share no common point and so their $(s+1)$-cliques in $\Gamma$ share no vertices and are therefore not adjacent in $C_{s+1}(\Gamma)$. So the collinearity graph of the dual of the $GQ(s,t)$ that formed $\Gamma$ is isomorphic to the $(s+1)$-clique graph of $\Gamma$.\\
    Now we will show that $C_{s+1}(\Gamma)$ is strongly regular. As above, it is well known that $\Gamma$ is an srg$((s+1)(st+1),s(t+1),s-1, t+1)$ so it has the spectrum \[s(t+1)^1,(s-1)^{\frac{st(t+1)(s+1)}{s+t}},(-t-1)^{\frac{s^2(st+1)}{s+t}}.\]
    Since $\Gamma$ is $(s+1)$-clique regular and a regular clique assembly from Corollary \ref{thm:gq}, and the smallest eigenvalue of $\Gamma$ is equal to $-\frac{s(t+1)}{(s+1)-1}=-t-1$, then Theorem \ref{thm:srg} tells us that $C_{s+1}(\Gamma)$ is strongly regular with parameters as above.
\end{proof}

\appendix
\titleformat{\chapter}[display]
  {\normalfont\bfseries\centering\Huge}
  {Appendix \thechapter} 
  {0em}
  {}
\chapter{Finding Feasible Locally Linear Strongly Regular Parameter Sets} \label{ap:llsrg}
As discussed in section \ref{sec:localsrgs}, to see if our necessary condition would eliminate any feasible locally linear strongly regular parameter sets, we needed a way to find a large number of these sets. Strongly regular graphs have four parameters $(n,k,\lambda,\mu)$, but they must satisfy the equation $(n-k-1)\mu=k(k-\lambda-1)$. This restriction reduces the feasible parameter space from four to three dimensions, and since the locally linear restriction also forces $\lambda=1$, we need only check in a two dimensional space.\par
This means, we could write an algorithm that uses nested loops to iterate over only two of the parameters $k$ and $n$, checks if the $\mu$ value implied by the equation\linebreak$(n-k-1)\mu=k(k-2)$ is a non-negative integer, and finally checks that the resulting parameter set's spectrum, given by equations (\ref{eq:srgspec}), are integral with positive integer multiplicities. Thus, finding all parameters with $k$ less than some large value $K$ in this manner, would have an average time complexity of $O(K^2)$.\par
We employ a method which is faster than this naive method, shown in Algorithm \ref{alg:localpara}. Instead of iterating over $k$ and $n$, we solve equation (\ref{eq:srgspec}) for $n=\frac{k(k-2)}{\mu}+k+1$, and notice this implies $\mu$ must divide $k(k-2)$ since $n$ must be an integer. So while our algorithm's outer loop still iterates over all $k$ less than $K$, the inner loop needs only to iterate over the $\mu$ which are divisors of $k(k-2)$ and also less than $k$.\par
To enumerate all possible values of $\mu$, we define a Python function which we will denote here as MuSet$(k)$. It takes as input an integer $k$, and outputs all the positive divisors of $k(k-2)$ which are less than $k$. The function primarily uses the prime factorization function \texttt{factorint} from the Python library SymPy, which for an integer $k$, finds its prime factorization with an average time complexity of $O(\sqrt{k})$ \cite{SymPy}. Since we need to find divisors of $k(k-2)$, we use \texttt{factorint} to find the prime factorizations of $k$ and $k-2$ individually which each have a runtime of $O(\sqrt{k})$. So factorizing $k(k-2)$ in this manner means MuSet$(k)$ has a runtime of $O(\sqrt{k})$.\par
Since on average for each $k$ up to $K$, the size of MuSet$(k)$ is $O(\log K)$ \linebreak \cite[Theorem 3.3]{apostol}, this is also the average time complexity of the Inner Loop in \linebreak Algorithm \ref{alg:localpara}. So we get that the runtime of the Outer Loop in Algorithm \ref{alg:localpara} is given by $O(K\sqrt{K})+O(K\log K)=O(K^{1.5})$, which is faster than the naive method's runtime of $O(K^2)$ as desired.

\begin{algorithm}[htpb]
\caption{Enumerate Feasible Locally Linear SRG Parameter Sets}
\label{alg:localpara}
\DontPrintSemicolon
\SetAlgoLined
\KwIn{Max value of $k$, $K$}
\KwOut{Set of Feasible Locally Linear SRG Parameter Sets with $k\leq K$}

FeasibleParameters $\gets \{\}$\;
\For{$k \gets 1$ \KwTo $K$ \hspace{1.25cm} \tcp{Outer Loop}} { 
    \For{$\mu \in$ \textnormal{MuSet}$(k)$ \quad \tcp{Inner Loop}} {
    $n\gets \frac{k(k-2)}{\mu}+k+1$\;
    \If{\textnormal{$r,s,f,g$ given by equations (\ref{eq:srgspec}) are integers and $f,g$ are positive}}{
            append $(n,k,1,\mu)$ to FeasibleParameters
        }
}}
\Return{\textnormal{FeasibleParameters}}\;
\end{algorithm}

\clearpage
\addcontentsline{toc}{chapter}{Bibliography}
\bibliographystyle{plainurl}
\bibliography{refs}
\end{document}